\documentclass{amsart}
\usepackage[margin=1in]{geometry}
\usepackage{color,amssymb, comment, mathrsfs, upgreek, contour, soul, cleveref, mathtools, enumerate, enumitem}
\usepackage[justification=centering]{caption}

\usepackage{tikz,tikz-cd}
\usetikzlibrary{cd, matrix, arrows, fit, shapes.geometric, patterns}
\tikzset{mynode/.style={ellipse, draw, inner xsep=0pt}}

\newlist{observations}{enumerate}{1}
\setlist[observations]{label=Cond. \Alph*., leftmargin=1in}
\Crefname{observations}{Cond. }{Obvservations}

\theoremstyle{definition}
\newtheorem{theorem}{Theorem}[section]
\newtheorem{proposition}[theorem]{Proposition}
\newtheorem{lemma}[theorem]{Lemma}

\newtheorem{chunk}[theorem]{}
\newtheorem{notation}[theorem]{Notation}
\newtheorem{setup}[theorem]{Setup}

\theoremstyle{remark}
\newtheorem{remark}[theorem]{Remark}

\newcommand{\EE}{\mathsf{E}}
\newcommand{\FF}{\mathsf{F}}
\newcommand{\GG}{\mathsf{G}}

\DeclareMathOperator{\Tor}{Tor}
\DeclareMathOperator{\Hom}{Hom}
\DeclareMathOperator{\im}{im}
\DeclareMathOperator{\rank}{rank}

\allowdisplaybreaks

\numberwithin{equation}{section}

\makeatletter
\newcommand\newref[1]{#1\def\@currentlabel{#1}}
\makeatother

\begin{document}

\title[Realizing Algebra Structures on Free Resolutions of Grade 3 Perfect Ideals]{Realizing Algebra Structures on Free Resolutions of Grade 3 Perfect Ideals}

\author[A.~Hardesty]{Alexis Hardesty}
\address{Division of Mathematics,
Texas Woman's University, Denton, TX 76204, U.S.A.}
\email{ahardesty1@twu.edu}


\begin{abstract}
Perfect ideals $I$ of grade $3$ in a local ring $(R,\mathfrak{m},\Bbbk)$ can be classified based on multiplicative structures on $\Tor^R_\bullet(R/I,\Bbbk)$. The classification is incomplete in the sense that it remains open which of the possible algebra structures actually occur; this \emph{realizability question} was formally posed by Avramov in 2012. Of five classes of algebra structures, the realizability question has been answered for one class. In this work, we answer the realizability question for two more classes and contribute towards an answer for a third.
\end{abstract}

\maketitle

\section{Introduction}
Let $(R,\mathfrak{m},\Bbbk)$ be a regular local ring and let $I\subseteq R$ be a perfect ideal of grade $3$. By a result of Buchsbaum and Eisenbud \cite[Proposition 1.1]{BE77}, the minimal free resolution $F_\bullet$ of $R/I$ over $R$ has a differential graded (DG) algebra structure. This induces a graded algebra structure on $\Tor^R_\bullet(R/I,\Bbbk)=\mathrm{H}_\bullet(F_\bullet\otimes_R\Bbbk)$.  Results of Weyman \cite[Theorem 4.1]{Weyman89} and of Avramov, Kustin, and Miller \cite[Theorem 2.1]{AKM88} show that this structure supports a classification scheme for grade $3$ perfect ideals in $R$. In particular, consider the minimal free resolution $F_{\bullet}$ of $R/I$:
\begin{equation*}
\begin{tikzcd} 
0 \ar[r] & F_{3} \ar[r,"d_3"] & F_{2} \ar[r,"d_2"] & F_{1} \ar[r,"d_1"] & R \ar[r] & 0.
\end{tikzcd}
\end{equation*}
Set $m = \rank_R(F_1)$ and $n = \rank_R(F_3)$. We say that $I$ has \emph{format} $(m,n)$. Denote $A_\bullet = \Tor_\bullet^R(R/I,\Bbbk)$ and set the following notation:
\[
p=\rank_\Bbbk A_1A_1, \quad q=\rank_\Bbbk A_1A_2, \quad r=\rank\delta\]
where $\delta: A_2\rightarrow \Hom_\Bbbk(A_1,A_3)$ is defined by $\delta(x)(y)=xy$ for $x\in A_2$ and $y\in A_1$. 

The classes in the scheme are named $\mathbf{B}$, $\mathbf{C}(3)$, $\mathbf{G}(r)$, $\mathbf{H}(p,q)$, and $\mathbf{T}$. As you can see by the naming convention, class $\mathbf{G}$ and class $\mathbf{H}$ are  parameterized families of classes depending on parameter $r$ and parameters $p$ and $q$, respectively. Other than these parameters, the values of $p$, $q$, and $r$ are fixed for each class. 

In \cite{Avramov12}, Avramov revisited this classification, keeping track of the values mentioned above, along with the \emph{Cohen-Macaulay defect}, a measure of how close quotient ring $R/I$ is to being Cohen-Macaulay. In this paper, we only consider the case when $R/I$ is Cohen-Macaulay, guaranteed by the assumptions that $I$ is perfect and $R$ is a regular ring. We address a reformulated version of the \emph{realizability question} first posed by Avramov in \cite[Question 3.8]{Avramov12}: which five-tuples $(m,n,p,q,r)$ are realized by some Cohen-Macaulay ring $R/I$?

In this paper, we consider the realizability question by fixing a format $(m,n)$ and then asking which classes (with fixed values $p$, $q$, and $r$) are actually realizable for this format. Authors have contributed towards the realizability question in two ways: proving a class is not realizable for particular formats or realizing a class for particular formats. If there is a result stating that a class is not realizable for a particular format, we say that this class is \emph{not permissible}; otherwise, we say the class is \emph{permissible.} Works by Brown \cite{Brown84}, Avramov \cite{Avramov81,Avramov12}, and Christensen, Veliche, and Weyman \cite{CVW17Licci,CVW19Trim,CVW20Linkage}, restrict the permissible classes for specific formats. \Cref{tab:1} illustrates the permissible classes $\mathbf{H}(p,q)$ for format $(m,n)=(8,6)$. Values of $p$ and $q$ for which the class $\mathbf{H}(p,q)$ is permissible are represented by white boxes and values of $p$ and $q$ for which the class $\mathbf{H}(p,q)$ is not permissible are represented by dotted boxes.

\begin{table}
\centering
\begin{tikzpicture}[x=1.1cm,y=.5cm]
\draw[step=1.0,gray,very thin](0,0) grid (9,8); 
\fill[pattern=crosshatch dots] (1,4) rectangle (2,5); 
\fill[pattern=crosshatch dots] (1,5) rectangle (2,6); 
\fill[pattern=crosshatch dots] (1,6) rectangle (2,7); 
\draw[black,thick](1,4) -- (2,4);
\draw[black,thick](2,4) -- (2,5);

\fill[pattern=crosshatch dots] (2,5) rectangle (3,6); 
\fill[pattern=crosshatch dots] (2,6) rectangle (3,7); 
\draw[black,thick](2,5) -- (3,5);
\draw[black,thick](3,4) -- (3,5);

\fill[pattern=crosshatch dots] (3,4) rectangle (4,5); 
\fill[pattern=crosshatch dots] (3,5) rectangle (4,6); 
\fill[pattern=crosshatch dots] (3,6) rectangle (4,7); 
\draw[black,thick](3,4) -- (4,4);
\draw[black,thick](4,4) -- (4,5);

\fill[pattern=crosshatch dots] (4,5) rectangle (5,6); 
\fill[pattern=crosshatch dots] (4,6) rectangle (5,7); 
\draw[black,thick](4,5) -- (5,5);
\draw[black,thick](5,4) -- (5,5);
      
\fill[pattern=crosshatch dots] (5,4) rectangle (6,5); 
\fill[pattern=crosshatch dots] (5,5) rectangle (6,6); 
\fill[pattern=crosshatch dots] (5,6) rectangle (6,7);
\draw[black,thick](5,4) -- (6,4);
\draw[black,thick](6,4) -- (6,5);

\fill[pattern=crosshatch dots] (6,1) rectangle (7,2); 
\fill[pattern=crosshatch dots] (6,3) rectangle (7,4); 
\fill[pattern=crosshatch dots] (6,5) rectangle (7,6); 
\fill[pattern=crosshatch dots] (6,6) rectangle (7,7);
\draw[black,thick](6,5) -- (7,5);
\draw[black,thick](6,4) -- (7,4);
\draw[black,thick](6,3) -- (7,3);
\draw[black,thick](6,2) -- (7,2);
\draw[black,thick](6,1) -- (7,1);
\draw[black,thick](7,4) -- (7,5); 
\draw[black,thick](7,2) -- (7,3);
\draw[black,thick](7,0) -- (7,1);
\draw[black,thick](6,3) -- (6,4);
\draw[black,thick](6,1) -- (6,2);

\fill[pattern=crosshatch dots] (7,0) rectangle (8,1); 
\fill[pattern=crosshatch dots] (7,1) rectangle (8,2); 
\fill[pattern=crosshatch dots] (7,2) rectangle (8,3); 
\fill[pattern=crosshatch dots] (7,3) rectangle (8,4); 
\fill[pattern=crosshatch dots] (7,4) rectangle (8,5); 
\fill[pattern=crosshatch dots] (7,5) rectangle (8,6); 
\fill[pattern=crosshatch dots] (7,6) rectangle (8,7); 
\draw[black,thick](8,6) -- (8,7);
\draw[black,thick](8,6) -- (9,6);

\fill[pattern=crosshatch dots] (8,0) rectangle (9,1); 
\fill[pattern=crosshatch dots] (8,1) rectangle (9,2); 
\fill[pattern=crosshatch dots] (8,2) rectangle (9,3); 
\fill[pattern=crosshatch dots] (8,3) rectangle (9,4); 
\fill[pattern=crosshatch dots] (8,4) rectangle (9,5); 
\fill[pattern=crosshatch dots] (8,5) rectangle (9,6); 

\draw[black,thick](0,0) -- (9,0);
\draw[black,thick](0,0) -- (0,8);
\draw[black,thick](0,8) -- (9,8);
\draw[black,thick](9,0) -- (9,8);
\draw[black,thick](1,0) -- (1,8);
\draw[black,thick](0,7) -- (9,7);


\draw (0.5,0.5) node {$q=0$}; 
\draw (0.5,1.5) node {$q=1$}; 
\draw (0.5,2.5) node {$q=2$};
\draw (0.5,3.5) node {$q=3$}; 
\draw (0.5,4.5) node {$q=4$}; 
\draw (0.5,5.5) node {$q=5$}; 
\draw (0.5,6.5) node {$q=6$};

\draw (1.5,7.5) node {$p=0$}; 
\draw (2.5,7.5) node {$p=1$}; 
\draw (3.5,7.5) node {$p=2$};
\draw (4.5,7.5) node {$p=3$}; 
\draw (5.5,7.5) node {$p=4$}; 
\draw (6.5,7.5) node {$p=5$}; 
\draw (7.5,7.5) node {$p=6$};
\draw (8.5,7.5) node {$p=7$};
\end{tikzpicture}
\caption{Permissible $\mathbf{H}(p,q)$ classes for format $(8,6)$.} \label{tab:1}
\end{table}

In this paper, we focus on realizing permissible classes using linkage of ideals. Linkage was used by Avramov, Kustin, and Miller in \cite{AKM88} to establish the classification scheme and is a main tool in several of the papers addressing realizability. In this work, we use linkage to prove existence of ideals of class $\mathbf{T}$, class $\mathbf{B}$, and class $\mathbf{H}$. More precisely, we show that for all formats $(m,n)$ with a permissible $\mathbf{T}$ class, a permissible $\mathbf{B}$ class, a permissible $\mathbf{H}(n-1,q)$ class, or a permissible $\mathbf{H}(p,m-4)$ class, there exists an ideal of this class. We will refer to the latter two classes as \emph{boundary classes}. 

\Cref{tab:2} illustrates the permissible boundary classes for format $(8,6)$. Values of $p$ and $q$ for which the class $\mathbf{H}(p,q)$ is a permissible boundary class are represented by black boxes, values of $p$ and $q$ for which the class $\mathbf{H}(p,q)$ is permissible but \emph{not} a boundary class are represented by white boxes, and values of $p$ and $q$ for which the class $\mathbf{H}(p,q)$ is not permissible are represented by dotted boxes. In particular, the classes $\mathbf{H}(5,0)$, $\mathbf{H}(5,2)$, and $\mathbf{H}(5,4)$ are permissible boundary classes of the form $\mathbf{H}(n-1,q)$. Moreover, the classes $\mathbf{H}(1,4)$, $\mathbf{H}(3,4)$, and $\mathbf{H}(5,4)$ permissible boundary classes of the form $\mathbf{H}(p,m-4)$.

\begin{table}
\centering
\begin{tikzpicture}[x=1.1cm,y=.5cm]
\draw[step=1.0,gray,very thin](0,0) grid (9,8); 

\fill[black] (2,4) rectangle (3,5); 
\fill[black] (4,4) rectangle (5,5);
\fill[black] (6,4) rectangle (7,5);
\fill[black] (6,2) rectangle (7,3);
\fill[black] (6,0) rectangle (7,1);

\fill[pattern=crosshatch dots] (1,4) rectangle (2,5); 
\fill[pattern=crosshatch dots] (1,5) rectangle (2,6); 
\fill[pattern=crosshatch dots] (1,6) rectangle (2,7); 
\draw[black,thick](1,4) -- (2,4);
\draw[black,thick](2,4) -- (2,5);

\fill[pattern=crosshatch dots] (2,5) rectangle (3,6); 
\fill[pattern=crosshatch dots] (2,6) rectangle (3,7); 
\draw[black,thick](2,5) -- (3,5);
\draw[black,thick](3,4) -- (3,5);

\fill[pattern=crosshatch dots] (3,4) rectangle (4,5); 
\fill[pattern=crosshatch dots] (3,5) rectangle (4,6); 
\fill[pattern=crosshatch dots] (3,6) rectangle (4,7); 
\draw[black,thick](3,4) -- (4,4);
\draw[black,thick](4,4) -- (4,5);

\fill[pattern=crosshatch dots] (4,5) rectangle (5,6); 
\fill[pattern=crosshatch dots] (4,6) rectangle (5,7); 
\draw[black,thick](4,5) -- (5,5);
\draw[black,thick](5,4) -- (5,5);
      
\fill[pattern=crosshatch dots] (5,4) rectangle (6,5); 
\fill[pattern=crosshatch dots] (5,5) rectangle (6,6); 
\fill[pattern=crosshatch dots] (5,6) rectangle (6,7);
\draw[black,thick](5,4) -- (6,4);
\draw[black,thick](6,4) -- (6,5);

\fill[pattern=crosshatch dots] (6,1) rectangle (7,2); 
\fill[pattern=crosshatch dots] (6,3) rectangle (7,4); 
\fill[pattern=crosshatch dots] (6,5) rectangle (7,6); 
\fill[pattern=crosshatch dots] (6,6) rectangle (7,7);
\draw[black,thick](6,5) -- (7,5);
\draw[black,thick](6,4) -- (7,4);
\draw[black,thick](6,3) -- (7,3);
\draw[black,thick](6,2) -- (7,2);
\draw[black,thick](6,1) -- (7,1);
\draw[black,thick](7,4) -- (7,5); 
\draw[black,thick](7,2) -- (7,3);
\draw[black,thick](7,0) -- (7,1);
\draw[black,thick](6,3) -- (6,4);
\draw[black,thick](6,1) -- (6,2);

\fill[pattern=crosshatch dots] (7,0) rectangle (8,1); 
\fill[pattern=crosshatch dots] (7,1) rectangle (8,2); 
\fill[pattern=crosshatch dots] (7,2) rectangle (8,3); 
\fill[pattern=crosshatch dots] (7,3) rectangle (8,4); 
\fill[pattern=crosshatch dots] (7,4) rectangle (8,5); 
\fill[pattern=crosshatch dots] (7,5) rectangle (8,6); 
\fill[pattern=crosshatch dots] (7,6) rectangle (8,7); 
\draw[black,thick](8,6) -- (8,7);
\draw[black,thick](8,6) -- (9,6);

\fill[pattern=crosshatch dots] (8,0) rectangle (9,1); 
\fill[pattern=crosshatch dots] (8,1) rectangle (9,2); 
\fill[pattern=crosshatch dots] (8,2) rectangle (9,3); 
\fill[pattern=crosshatch dots] (8,3) rectangle (9,4); 
\fill[pattern=crosshatch dots] (8,4) rectangle (9,5); 
\fill[pattern=crosshatch dots] (8,5) rectangle (9,6); 

\draw[black,thick](0,0) -- (9,0);
\draw[black,thick](0,0) -- (0,8);
\draw[black,thick](0,8) -- (9,8);
\draw[black,thick](9,0) -- (9,8);
\draw[black,thick](1,0) -- (1,8);
\draw[black,thick](0,7) -- (9,7);


\draw (0.5,0.5) node {$q=0$}; 
\draw (0.5,1.5) node {$q=1$}; 
\draw (0.5,2.5) node {$q=2$};
\draw (0.5,3.5) node {$q=3$}; 
\draw (0.5,4.5) node {$q=4$}; 
\draw (0.5,5.5) node {$q=5$}; 
\draw (0.5,6.5) node {$q=6$};

\draw (1.5,7.5) node {$p=0$}; 
\draw (2.5,7.5) node {$p=1$}; 
\draw (3.5,7.5) node {$p=2$};
\draw (4.5,7.5) node {$p=3$}; 
\draw (5.5,7.5) node {$p=4$}; 
\draw (6.5,7.5) node {$p=5$}; 
\draw (7.5,7.5) node {$p=6$};
\draw (8.5,7.5) node {$p=7$};
\end{tikzpicture}
\caption{Permissible $\mathbf{H}(p,q)$ boundary classes for format $(8,6)$}  \label{tab:2}
\end{table}

The paper is organized as follows. In Section 2 we recall background information on perfect ideals of grade 3 and their classification. We then state linkage results that connect an ideal of a particular class to an ideal of another, typically different, class. The proofs of these results are technical in nature and are therefore recorded in the Appendix. In Section 3, 4, and 5, we apply the results to construct ideals that realize all permissible $\mathbf{T}$ classes, all permissible $\mathbf{B}$ classes, and all permissible $\mathbf{H}(p,q)$ boundary classes, respectively.

\section{Background and Main Results}

Throughout the paper, let $(R,\mathfrak{m},\Bbbk)$ be a regular local ring and $I\subseteq\mathfrak{m}^2$ a perfect ideal of grade 3. Consider a minimal free resolution $F_{\bullet}$ of $R/I$. We can write $F_{\bullet}$ as 
\begin{equation*}
\begin{tikzcd} 
0 \ar[r] & F_{3} \ar[r,"d_3"] & F_{2} \ar[r,"d_2"] & F_{1} \ar[r,"d_1"] & R \ar[r] & 0.
\end{tikzcd}
\end{equation*}
Set $m = \rank_R(F_1)$ and $n = \rank_R(F_3)$. Using the terminology in \cite{CVW17Licci}, one may call the tuple $(1,m,m+n-1,n)$ the \emph{format} of $I$. Since the Euler characteristic of the resolution is zero, the size of the resolution is completely determined by $m$ and $n$, so we instead refer to the tuple $(m,n)$ as the \emph{format} of $I$ and write $F_\bullet$ as
\begin{equation*}
\begin{tikzcd} 
0 \ar[r] & R^n \ar[r,"d_3"] & R^{m+n-1} \ar[r,"d_2"] & R^m \ar[r,"d_1"] & R \ar[r] & 0.
\end{tikzcd}
\end{equation*}
Now, we look at closer at the product on $A_\bullet = \Tor_\bullet^R(R/I,\Bbbk)$ as determined in \cite{AKM88}.

\begin{chunk}\label{classification}
There exist bases \[
\{\mathsf{e}_i\}_{i=1,\ldots,m},\quad\{\mathsf{f}_i\}_{i=1,\ldots,m+n-1},\quad\{\mathsf{g}_i\}_{i=1,\ldots,n}
\]
of $A_1$, $A_2$, and $A_3$, respectively, such that the multiplicative structure on $\mathbf{A_\bullet}$ is one of the following:
\begin{equation*}
\begin{tabular}{rlrl}
    $\mathbf{C}(3)$: & $\mathsf{e}_1 \mathsf{e}_2 = \mathsf{f}_3$, $\mathsf{e}_2 \mathsf{e}_3 = \mathsf{f}_1$, $\mathsf{e}_3 \mathsf{e}_1 = \mathsf{f}_2$ & $\mathsf{e}_i \mathsf{f}_i = \mathsf{g}_1$ & for $1\leq i\leq 3$\\
    $\mathbf{T}$: & $\mathsf{e}_1 \mathsf{e}_2 = \mathsf{f}_3$, $\mathsf{e}_2 \mathsf{e}_3 = \mathsf{f}_1$, $\mathsf{e}_3 \mathsf{e}_1 = \mathsf{f}_2$\\
    $\mathbf{B}$: & $\mathsf{e}_1 \mathsf{e}_2 = \mathsf{f}_3$ & $\mathsf{e}_i \mathsf{f}_i = \mathsf{g}_1$ & for $1\leq i \leq 2$\\
    $\mathbf{G}(r)$: &  & $\mathsf{e}_i \mathsf{f}_i = \mathsf{g}_1$ & for $1\leq i\leq r$\\
    $\mathbf{H}(p,q)$: & $\mathsf{e}_i \mathsf{e}_{p+1} = \mathsf{f}_i$ for $1\leq i\leq p$ & $\mathsf{e}_{p+1}\mathsf{f}_{p+i} = \mathsf{g}_i$ & for $1\leq i\leq q$
\end{tabular}
\end{equation*}
All graded-commutativity rules are understood, and all other products not listed are zero. From the above table, it is clear that the values of $p$, $q$ and $r$ for each class are:
\begin{equation*}
\begin{tabular}{rllll}
    $\mathbf{C}(3)$: & & $p=3$ & $q=1$ & $r=3$ \\
    $\mathbf{T}$: & & $p=3$ & $q=0$ & $r=0$\\
    $\mathbf{B}$: & & $p=1$ & $q=1$ & $r=2$\\
    $\mathbf{G}(r)$: & & $p=0$ & $q=1$ & \\
    $\mathbf{H}(p,q)$: & &   &   & $r=q$
\end{tabular}
\end{equation*}
\end{chunk}

The classification above was proved in \cite[Theorem 4.1]{Weyman89} and \cite[Theorem 2.1]{AKM88}. In this work, we refer to \cite{AKM88} when referencing the classification. It was shown in \cite{AKM88} that, even though the DG algebra structure of $F_\bullet$ is not unique, the induced algebra structure on $A_\bullet$ is unique.

Notice that the classes $\mathbf{G}(0)$ and $\mathbf{G}(1)$ coincide with $\mathbf{H}(0,0)$ and $\mathbf{H}(0,1)$, respectively. For this reason, when we discuss class $\mathbf{G}(r)$, we will assume that $r\geq 2$. 

Below we list the main tools of the paper. The proofs of these results are recorded in the appendix. 

\begin{theorem}\label{linktoT}
Let $I$ have format $(m,n)$. Then $I$ is linked to an ideal of class $\mathbf{T}$ and format $(n+3,m)$.
\end{theorem}

\begin{theorem}\label{linkT}
If $I$ is of class $\mathbf{T}$ and format $(m,n)$, then $I$ is linked to ideals described by the following data:
\begin{center}
\begin{tabular}{rcc}
& class & format\\
(i) & $\mathbf{H}(2,0)$ & $(n+3,m-1)$\\
(ii) & $\mathbf{H}(2,2)$ & $(n+3,m-1)$\\
(iii) & $\mathbf{H}(1,2)$ & $(n+3,m-2)$\\
(iv) & $\mathbf{B}$ & $(n+2,m-2)$
\end{tabular}
\end{center}
\end{theorem}

\begin{theorem}\label{linkG}
If $I$ is of class $\mathbf{G}(r)$ and format $(m,n)$, then $I$ is linked to ideals described by the following data:
\begin{center}
\begin{tabular}{rcc}
& class & format\\
(i) & $\mathbf{H}(3,0)$ & $(n+3,m-1)$\label{linkGup4H30}\\
(ii) & $\mathbf{T}$ & $(n+3,m-2)$\label{linkGup2T}
\end{tabular}
\end{center}
\end{theorem}

\begin{theorem}\label{linkH}
If $I$ is of class $\mathbf{H}(p,q)$ and format $(m,n)$, then $I$ is linked to ideals described by the following data under the following assumptions:
\begin{center}
\begin{tabular}{rccl}
& class & format & assumptions on $p$\\
(i) & $\mathbf{H}(2,1)$ & $(n+3,m-1)$ & $1\leq p$\\
(ii) & $\mathbf{H}(q+2,p)$ & $(n+3,m-1)$ & \\
(iii) & $\mathbf{H}(1,1)$ & $(n+3,m-2)$ & $1\leq p\leq m-2$\\
(iv) & $\mathbf{H}(q+1,p)$ & $(n+3,m-2)$ & $0\leq p\leq m-2$\\
(v) & $\mathbf{H}(0,p)$ & $(n+3,m-3)$ & $2\leq p\leq m-3$
\end{tabular}
\end{center}
\end{theorem}

\begin{chunk}\label{realizeGor}
Recall that an ideal is called \emph{Gorenstein of grade g} provided it is perfect of grade $g$ with format $(m,1)$. Let $m\geq 3$. For format $(m,1)$, there are no permissible classes when $m$ is even; see \cite[Corollary 2.2]{BE77}. When $m=3$, the only permissible class is $\mathbf{C}(3)$ and when $m\geq 5$ is odd, the only permissible class is $\mathbf{G}(m)$; see Remark (ii) after \cite[Definition 2.2]{AKM88}. These classes can all be realized; see \cite[3.9.1]{Avramov12}.

Note that there are ideals of class $\mathbf{G}(r)$ that are not Gorenstein; they have $n\geq 2$. These ideals are not considered in this paper, but the realizability question remains open in this case. Sporadic results are found in works by Christensen, Veliche, and Weyman \cite{CVW19Trim}, VandeBogert \cite{Vandebogert20Trim}, Ferraro and Hardesty \cite{FH23}, and Ferraro and Moore \cite{FM24}.
\end{chunk}

\begin{chunk}\label{realizeHS}
Let $p\geq 3$. Class $\mathbf{H}(p,p-1)$ has a permissible format $(p+1,p-1)$. Ideals with this class and format are known as hypersurface sections and can all be realized; see \cite[3.9.2(b)]{Avramov12}.
\end{chunk}

For the sake of completeness, we include the following lemmas that realize permissible formats for two families of classes that have already been realized in the literature. The first family is known as almost complete intersections and was first realized in \cite[Remark (1)]{Avramov81}; see also \cite[Remark 4.2]{CVW20Linkage} in light of \Cref{realizeGor} above. The second family was realized in \cite[Theorem 4.4]{Brown84}. We include proofs that use Theorems \ref{linkT}-\ref{linkH} above. The arguments in subsequent sections will follow a similar style. 

\begin{lemma}\label{lem:RealizeACI}
Under the following assumptions on $n$, one can realize the following classes associated with the given formats:
\begin{center}
\begin{tabular}{rccc}
& class & format & assumptions on $n$\\
(i) & $\mathbf{H}(3,2)$ & $(4,2)$ \\
(ii) & $\mathbf{H}(3,0)$ & $(4,n)$ & $n\geq 4$ even \label{ACIii}\\
(iii) & $\mathbf{T}$ & $(4,n)$ & $n\geq 3$ odd \label{ACIiii}\\
\end{tabular}
\end{center}
\end{lemma}

\begin{proof}
Let $I$ denote a Gorenstein ideal of class $\mathbf{G}(m)$ and format $(m,1)$ for some $m\geq 5$ odd as in \Cref{realizeGor}. Let $m=5$. By \cite[Proposition 3.1(b)]{CVW20Linkage}, $I$ is linked to an ideal $I'$ of class $\mathbf{H}(p,q)$ and format $(4,2)$. By \cite[Theorem 4.1]{CVW20Linkage}, $I'$ must have class $\mathbf{H}(3,2)$, realizing (i). By \Cref{linkG}(i), $I$ is linked to an ideal of class $\mathbf{H}(3,0)$ and format $(4,m-1)$ for $m-1\geq 4$ even, realizing (ii). By \Cref{linkG}(ii), $I$ is linked to an ideal of class $\mathbf{T}$ and format $(4,m-2)$ for $m-2\geq 3$ odd, realizing (iii).
\end{proof}

\begin{lemma}\label{lem:RealizeType2}
Under the following assumptions on $m$, one can realize the following classes associated with the given formats:
\begin{center}
\begin{tabular}{rccc}
& class & format & assumptions on $m$\\
(i) & $\mathbf{H}(1,2)$ & $(m,2)$ & $m\geq 6$ even\\
(ii) & $\mathbf{B}$ & $(m,2)$ & $m\geq 5$ odd\\
\end{tabular}
\end{center}
\end{lemma}

\begin{proof}
Let $I$ denote an ideal of class $\mathbf{T}$ and format $(4,n)$ for some $n\geq 3$ odd, as constructed in \Cref{lem:RealizeACI}(iii). By \Cref{linkT}(iii), $I$ is linked to an ideal of class $\mathbf{H}(1,2)$ and format $(n+3,2)$ for $n+3\geq 6$ even, realizing (i). By \Cref{linkT}(iv), $I$ is linked to an ideal of class $\mathbf{B}$ and format $(n+2,2)$ for $n+2\geq 5$ odd, realizing (ii).
\end{proof}

\section{Class $\mathbf{T}$}

In this section, we survey previous results on realizability of class $\mathbf{T}$ and show that all permissible formats for class $\mathbf{T}$ are realizable.

\begin{lemma}\label{lem:linktoT}
If there exists an ideal of format $(m,n)$, then there exists an ideal of class $\mathbf{T}$ and format $(m+3,n+3)$. 
\end{lemma}

\begin{proof}
By two applications of \Cref{linktoT}, if $I$ is an ideal of format $(m,n)$, then $I$ linked to an ideal of class $\mathbf{T}$ and format $(m+3,n+3)$.
\end{proof}

\begin{theorem}\label{thm:realizeT}
For a format $(m,n)$ there exists an ideal of class $\mathbf{T}$ if and only if 
\begin{enumerate}[label=\roman*.]
    \item $m=4$ and $n\geq 3$ is odd, or
    \item $m\geq 5$ and $n\geq 4$.
\end{enumerate}
\end{theorem}

\begin{proof}
First, we show that the formats listed in the statement are indeed the only permissible formats for class $\mathbf{T}$. Let $I$ denote an ideal of class $\mathbf{T}$ and format $(m,n)$. By \cite[3.4.1(a)]{Avramov12} we must have $m\geq 4$. If $m=4$, then $n\geq 3$ is odd; see \cite[3.4.2]{Avramov12}. If $m\geq 5$, then by \Cref{realizeGor}, \cite[Corollary 4.5]{Brown84}, and by \cite[Corollary to Theorem 2.1]{Sanchez89} $n\geq 4$; see also \cite[7.2]{CVW20Linkage}.

Now, we show that there exists an ideal of class $\mathbf{T}$ for each of the formats listed in the statement. \Cref{lem:RealizeACI} shows that for all formats $(4,n)$ with $n\geq 3$ odd, there exists an ideal of class $\mathbf{T}$. 

We prove that there exist ideals of class $\mathbf{T}$ with the following formats for integers $m\geq5$, $n\geq 4$, and $k\geq 1$:
\begin{align}
(3k+2,n) && n\geq 3k+2\geq 5 \label{classTn2}\\
(m,3k+2) && m\geq 3k+5\geq 8 \label{classTm2}\\
(3k+3,n) && n\geq 3k+3\geq 6 \label{classTn0}\\
(m,3k+3) && m\geq 3k+6\geq 9 \label{classTm0}\\
(3k+4,n) && n\geq 3k+3\geq 6 \label{classTn1}\\
(m,3k+1) && m\geq 3k+3\geq 6 \label{classTm1}
\end{align}

First, we realize the base case $k=1$ for each of the formats \eqref{classTn2}-\eqref{classTm1}, which are the formats:
\begin{align*}
(5,n), n\geq 5 \qquad &(m,5), m\geq 8 \qquad 
(6,n), n\geq 6 \\ (m,6), m\geq 9 \qquad
&(7,n), n\geq 6 \qquad (m,4), m\geq 6.
\end{align*}

Let $I$ denote an ideal with format $(m,2)$ for $m\geq 5$ as constructed in \Cref{lem:RealizeType2}. Let $J$ denote an ideal with format $(m,3)$ for $m\geq 6$. When $m=6$, such an ideal is constructed in \cite[Theorem 2]{CV14Ex} and when $m\geq 7$, such an ideal is constructed in \cite[Corollary 5.9]{Vandebogert20Trim}. Let $K$ denote an almost complete intersection with format $(4,n)$ for some $n\geq 3$ as constructed in \Cref{lem:RealizeACI}. The first, third, and sixth base cases are realized by applying \Cref{linktoT} to ideals $I$, $J$, and $K$, respectively. The second, fourth, and fifth base cases are realized by applying \Cref{lem:linktoT} to ideals $I$, $J$, and $K$, respectively.

Now, we proceed with the induction step and assume $k>1$. This will be the same for all of the families \eqref{classTn2}-\eqref{classTm1}, so we only record the proof for \eqref{classTn2}. Assume we can realize an ideal of class $\mathbf{T}$ with format $(3k+2,n)$, where $n\geq 3k+2$. By \Cref{lem:linktoT}, there exists an ideal of class $\mathbf{T}$ with format $(3k+5,n+3)=(3(k+1)+2,n+3)$, where $n+3\geq 3k+5=3(k+1)+2$. This realizes the next ideal in the family, completing the induction step. 

Now, we summarize the ideals realized above to illustrate the steps needed to complete the proof. 

Realizability of ideals of class $\mathbf{T}$ and with formats listed in the families \eqref{classTn2}, \eqref{classTn0}, and \eqref{classTn1} above is equivalent to the realizability of ideals of class $\mathbf{T}$ and formats 
\begin{itemize}
    \item $(m,n)$ with $5\leq m\leq n$, and 
    \item $(m,m-1)$ with $7\leq m$ and $m\equiv_3 1$.
\end{itemize}

Moreover, realizability of ideals of class $\mathbf{T}$ and with formats listed in the families \eqref{classTm2}, \eqref{classTm0}, and \eqref{classTm1} above is equivalent to the realizability of ideals of class $\mathbf{T}$ and formats 
\begin{itemize}
    \item $(m,n)$ with $4\leq n\leq m-3$, and 
    \item $(n+2,n)$ with $4\leq n$ and $n\equiv_3 1$. 
\end{itemize}

We visualize these realizations for $5\leq m\leq 12$ and $4\leq n\leq 10$ in \Cref{tab:3}. Formats that have been realized are represented by boxes with a number inside. The number corresponds to the family realizing that format as outlined above. Notice that the bottom bullets above are simply a subset of the formats listed in the families \eqref{classTn1} and \eqref{classTm1}. The formats realized by the bottom bullets are bold in \Cref{tab:3} below.

\begin{table}[h]
\centering
\begin{tikzpicture}[x=1.3cm,y=.5cm]
\draw[step=1.0,gray,very thin](0,0) grid (9,8); 

\draw[black,thick](0,0) -- (9,0);
\draw[black,thick](0,0) -- (0,8);
\draw[black,thick](0,8) -- (9,8);
\draw[black,thick](9,0) -- (9,8);
\draw[black,thick](1,0) -- (1,8);
\draw[black,thick](0,7) -- (9,7);

\draw (0.5,6.5) node {$n=10$}; 
\draw (0.5,5.5) node {$n=9$}; 
\draw (0.5,4.5) node {$n=8$};
\draw (0.5,3.5) node {$n=7$}; 
\draw (0.5,2.5) node {$n=6$}; 
\draw (0.5,1.5) node {$n=5$}; 
\draw (0.5,0.5) node {$n=4$};

\draw (1.5,7.5) node {$m=5$}; 
\draw (2.5,7.5) node {$m=6$}; 
\draw (3.5,7.5) node {$m=7$};
\draw (4.5,7.5) node {$m=8$}; 
\draw (5.5,7.5) node {$m=9$}; 
\draw (6.5,7.5) node {$m=10$}; 
\draw (7.5,7.5) node {$m=11$};
\draw (8.5,7.5) node {$m=12$};

\draw (1.5,6.5) node {(\ref{classTm2})};
\draw (1.5,5.5) node {(\ref{classTm2})};
\draw (1.5,4.5) node {(\ref{classTm2})};
\draw (1.5,3.5) node {(\ref{classTm2})};
\draw (1.5,2.5) node {(\ref{classTm2})};
\draw (1.5,1.5) node {(\ref{classTm2})};

\draw (2.5,6.5) node {(\ref{classTm0})};
\draw (2.5,5.5) node {(\ref{classTm0})};
\draw (2.5,4.5) node {(\ref{classTm0})};
\draw (2.5,3.5) node {(\ref{classTm0})};
\draw (2.5,2.5) node {(\ref{classTm0})};

\draw (3.5,6.5) node {(\ref{classTm1})};
\draw (3.5,5.5) node {(\ref{classTm1})};
\draw (3.5,4.5) node {(\ref{classTm1})};
\draw (3.5,3.5) node {(\ref{classTm1})};
\draw (3.5,2.5) node {\textbf{(\ref{classTm1})}};

\draw (4.5,6.5) node {(\ref{classTm2})};
\draw (4.5,5.5) node {(\ref{classTm2})};
\draw (4.5,4.5) node {(\ref{classTm2})};

\draw (5.5,6.5) node {(\ref{classTm0})};
\draw (5.5,5.5) node {(\ref{classTm0})};

\draw (6.5,6.5) node {(\ref{classTm1})};
\draw (6.5,5.5) node {\textbf{(\ref{classTm1})}};

\draw (2.5,0.5) node {\textbf{(\ref{classTn1})}};
\draw (3.5,0.5) node {(\ref{classTn1})};
\draw (4.5,0.5) node {(\ref{classTn1})};
\draw (5.5,0.5) node {(\ref{classTn1})};
\draw (6.5,0.5) node {(\ref{classTn1})};
\draw (7.5,0.5) node {(\ref{classTn1})};
\draw (8.5,0.5) node {(\ref{classTn1})};

\draw (4.5,1.5) node {(\ref{classTn2})};
\draw (5.5,1.5) node {(\ref{classTn2})};
\draw (6.5,1.5) node {(\ref{classTn2})};
\draw (7.5,1.5) node {(\ref{classTn2})};
\draw (8.5,1.5) node {(\ref{classTn2})};

\draw (5.5,2.5) node {(\ref{classTn0})};
\draw (6.5,2.5) node {(\ref{classTn0})};
\draw (7.5,2.5) node {(\ref{classTn0})};
\draw (8.5,2.5) node {(\ref{classTn0})};

\draw (5.5,3.5) node {\textbf{(\ref{classTn1})}};
\draw (6.5,3.5) node {(\ref{classTn1})};
\draw (7.5,3.5) node {(\ref{classTn1})};
\draw (8.5,3.5) node {(\ref{classTn1})};

\draw (7.5,4.5) node {(\ref{classTn2})};
\draw (8.5,4.5) node {(\ref{classTn2})};

\draw (8.5,5.5) node {(\ref{classTn0})};

\end{tikzpicture}
\caption{Family of ideals realizing format $(m,n)$.}  \label{tab:3}
\end{table}

Therefore, it suffices to prove the realizability of ideals of class $\mathbf{T}$ and formats 
\begin{itemize}
\item $(m,m-1)$ with $5\leq m$, and 
\item $(m,m-2)$ with $7\leq m$.
\end{itemize}

Let $I$ denote a hypersurface section with format $(p+1,p-1)$ for $p\geq 3$ as constructed in \Cref{realizeHS}. By \Cref{linktoT}, $I$ is linked to an ideal of class $\mathbf{T}$ and format $(p+2,p+1)$ for $p+2\geq 5$. By \Cref{lem:linktoT}, $I$ is linked to an ideal of class $\mathbf{T}$ and format $(p+4,p+2)$ for $p+4\geq 7$. These realizations account for the remaining formats outlined above.
\end{proof}

\section{Class $\mathbf{B}$}

In this section, we survey previous results on realizability of class $\mathbf{B}$ and show that all permissible formats for class $\mathbf{B}$ are realizable. 

\begin{theorem}\label{thm:realizeB}
For each format $(m,n)$ there exists an ideal of class $\mathbf{B}$ if and only if 
\begin{enumerate}[label=\roman*.]
    \item $m\geq 5$ is odd and $n=2$, or
    \item $m\geq 6$ and $n\geq 3$.
\end{enumerate}
\end{theorem}

\begin{proof}
First, we show that the formats listed in the statement are indeed the only permissible formats for class $\mathbf{B}$. Let $I$ denote an ideal of class $\mathbf{B}$ and with format $(m,n)$. By \cite[3.4.1]{Avramov12} and \cite[3.4.2]{Avramov12} we must have $m\geq 5$ and by \Cref{realizeGor} we must have $n\geq 2$. Moreover, if $m=5$, then $n=2$; see \cite[Theorem 4.5]{CVW20Linkage}, and if $n=2$, then $m$ is odd, first proved in \cite[Theorem 4.4]{Brown84}; see also \cite[3.4.3]{AKM88}.

Now, we show that there exists an ideal of class $\mathbf{B}$ for each of the formats listed in the statement. \Cref{lem:RealizeType2} shows that we can realize all permissible formats for class $\mathbf{B}$ with $n=2$. It remains to show realizability for $m\geq 6$ and $n\geq 3$.

Let $I$ denote an ideal of class $\mathbf{T}$ with format $(m,n)$ for some $m\geq 5$ and $n\geq 4$, realized in \Cref{thm:realizeT}. By \Cref{linkT}(iv), $I$ is linked to an ideal of class $\mathbf{B}$ with format $(n+2,m-2)$. Since $n+2\geq 6$ and $m-2\geq 3$, we obtain ideals of class $\mathbf{B}$ within the stated bounds.
\end{proof}
 
\section{Class $\mathbf{H}$}

In this section, we survey previous results on realizability of class $\mathbf{H}$ and realize a family of classes not yet constructed in the literature, the \emph{boundary classes} $\mathbf{H}(n-1,q)$ and $\mathbf{H}(p,m-4)$. The following remark outlines the permissible boundary classes.

\begin{remark}\label{rmk:Hoptions}
Let $I$ be a perfect ideal of class $\mathbf{H}$ and format $(m,n)$. By \cite[3.4.1]{Avramov12} we must have $m\geq 4$ and by \Cref{realizeGor} we must have $n\geq 2$. The following restrictions on $p$ and $q$ are known given the values of $m$ and $n$:

\begin{enumerate}[label=\arabic*.]
    \item $p\leq\min\{m-1,n+1\}$, by \cite[Theorem 3.1]{Avramov12}. \label{eq:pMax}
    \item $q\leq\min\{m-2,n\}$, by \cite[Theorem 3.1]{Avramov12}.\label{eq:qMax}
\end{enumerate}
The following result shows that these upper bounds are achieved simultaneously:
\begin{enumerate}[label=\arabic*.]
\setcounter{enumi}{2}
    \item By \cite[Corollary 3.3]{Avramov12}.\label{eq:pqMax}The following are equivalent: 
    \[
    p=n+1 \textnormal{ and } q=m-2,\qquad p=m-1,\qquad q=n
    \]
\end{enumerate}
Ideals with the above values for $p$ and $q$ are hypersurface sections, discussed in \Cref{realizeHS}. Since these highest values of $p$ and $q$ occur only when $m=n+2$, we do not consider them the ``boundary" for the values of $p$ and $q$ for all formats $(m,n)$. Instead, we consider the boundary to be the next highest bounds on $p$ and $q$, which are indeed the highest bounds for formats other than $(m,m-2)$:
\begin{enumerate}[label=\arabic*.]
\setcounter{enumi}{3}
    \item If $p\neq n+1$, then $p\leq n-1$, by \cite[Theorem 1.1]{CVW20Linkage}. \label{eq:pMax2}
    \item If $q\neq m-2$, then $q\leq m-4$ by \cite[Theorem 1.1]{CVW20Linkage}.\label{eq:qMax2}
\end{enumerate}

Therefore, we consider classes with the values $p=n-1$ and $q=m-4$ the \emph{boundary classes} for a given format $(m,n)$. In this case, we have additional restrictions:
\begin{enumerate}[label=\arabic*.]
\setcounter{enumi}{5}
    \item If $p=n-1$, then $q\equiv_2 m-4$, by \cite[Theorem 1.1]{CVW20Linkage}. \label{eq:pBoundary}
    \item If $q=m-4$, then $p\equiv_2 n-1$, by \cite[Theorem 1.1]{CVW20Linkage}. \label{eq:qBoundary}
\end{enumerate}

In the cases $m=4$ or $n=2$, there are further restrictions; see \Cref{prop:Hm4n2}. Therefore, when $m\geq 5$ and $n\geq 3$, the permissible boundary classes are:

\begin{center}
\begin{tabular}{cccc}
class & format & restrictions\\
$\mathbf{H}(n-1,q)$ & $(m,n)$ & $q\leq\min\{m-4,n\}  \text{ and } q\equiv_2 m-4$\\
$\mathbf{H}(p,m-4)$ & $(m,n)$ & $p\leq\min\{m-1,n-1\}  \text{ and } p\equiv_2 n-1$
\end{tabular}
\end{center}
\end{remark}

\begin{theorem}\label{thm:realizeH}
For each format $(m,n)$ with $m\geq 5$ and $n\geq 3$, one can realize ideals of class $\mathbf{H}$ as follows: 
\begin{align*}
&\mathbf{H}(n-1,q), \text{ with } q\leq\min\{m-4,n\}  \text{ and } q\equiv_2 m-4\\
&\mathbf{H}(p,m-4), \text{ with } p\leq\min\{m-1,n-1\}  \text{ and } p\equiv_2 n-1
\end{align*}
That is, one can realize all permissible boundary classes for formats $(m,n)$ with $m\geq 5$ and $n\geq 3$.
\end{theorem}

\begin{proof} 
\newcounter{Hproof}
First, we realize the permissible boundary classes for $n=3$ and $m=5$. Let $n=3$. The boundary classes have $p=n-1=2$. By \Cref{rmk:Hoptions}, we have that $q\leq \min\{m-4,3\}$ and $q\equiv_2 m-4$, so we must realize the following classes:
\addtocounter{Hproof}{1}
\begin{center}
\begin{tabular}{lcccccc}
\newref{\theHproof}$(q)$\label{n=3} &\quad& {class} &\quad& \textnormal{format} &\quad& \textnormal{parameters}\\
\newref{\theHproof}$(0)$\label{n=3(0)} &\quad& $\mathbf{H}(2,0)$ &\quad& $(m,3)$ &\quad& $m\geq 6$ even\\
\newref{\theHproof}$(1)$\label{n=3(1)} &\quad& $\mathbf{H}(2,1)$ &\quad& $(m,3)$ &\quad& $m\geq 5$ odd\\
\newref{\theHproof}$(2)$\label{n=3(2)} &\quad& $\mathbf{H}(2,2)$ &\quad& $(m,3)$ &\quad& $m\geq 6$ even\\
\newref{\theHproof}$(3)$\label{n=3(3)} &\quad& $\mathbf{H}(2,3)$ &\quad& $(m,3)$ &\quad& $m\geq 7$ odd
\end{tabular}
\end{center}

Let $I$ denote an ideal of class $\mathbf{T}$ and format $(4,t)$ for $t\geq 3$ odd, listed in \Cref{lem:RealizeACI}(iii). By \Cref{linkT}(i), $I$ is linked to an ideal of class $\mathbf{H}(2,0)$ and format $(t+3,3)$, realizing the family \ref{n=3}$(0)$. By \Cref{linkT}(ii), $I$ is also linked to an ideal of class $\mathbf{H}(2,2)$ and format $(t+3,3)$, realizing the family \ref{n=3}$(2)$. 

Let $J$ denote an ideal of class $\mathbf{H}(3,2)$ and format $(4,2)$, listed in \Cref{lem:RealizeACI}(i). By \Cref{linkH}(i), $J$ is linked to an ideal of class $\mathbf{H}(2,1)$ and format $(5,3)$, the first ideal in the family \ref{n=3}$(1)$. 

Let $K$ denote an ideal of class $\mathbf{H}(3,0)$ and format $(4,t)$ for $t\geq 4$ even, listed in \Cref{lem:RealizeACI}(ii). By \Cref{linkH}(i), $K$ is linked to an ideal of class $\mathbf{H}(2,1)$ and format $(t+3,3)$, realizing the remaining classes in the family \ref{n=3}$(1)$. By \Cref{linkH}(ii), $K$ is also linked to an ideal of class  $\mathbf{H}(2,3)$ and format $(t+3,3)$, realizing the family \ref{n=3}$(3)$. 

Now we set $m=5$. The boundary classes have $q=m-4=1$. By \Cref{rmk:Hoptions}, we have that $p\leq \min\{n-1,4\}$ and $p\equiv_2 n-1$, so we must realize the following classes:
\addtocounter{Hproof}{1}
\begin{center}
\begin{tabular}{lcccccc}
\newref{\theHproof}$(p)$\label{m=5} &\quad& {class} &\quad& \textnormal{format} &\quad& \textnormal{parameters}\\
\newref{\theHproof}$(0)$\label{m=5(0)} &\quad& $\mathbf{H}(0,1)$ &\quad& $(5,n)$ &\quad& $n\geq 3$ odd\\
\newref{\theHproof}$(1)$\label{m=5(1)} &\quad& $\mathbf{H}(1,1)$ &\quad& $(5,n)$ &\quad& $n\geq 4$ even\\
\newref{\theHproof}$(2)$\label{m=5(2)} &\quad& $\mathbf{H}(2,1)$ &\quad& $(5,n)$ &\quad& $n\geq 3$ odd\\
\newref{\theHproof}$(3)$\label{m=5(3)} &\quad& $\mathbf{H}(3,1)$ &\quad& $(5,n)$ &\quad& $n\geq 4$ even\\
\newref{\theHproof}$(4)$\label{m=5(4)} &\quad& $\mathbf{H}(4,1)$ &\quad& $(5,n)$ &\quad& $n\geq 5$ odd
\end{tabular}
\end{center}

Let $I$ denote an ideal of class $\mathbf{H}(2,0)$ and format $(s,3)$ for $s\geq 6$ even, listed in \ref{n=3}$(0)$. By \cite[Proposition 3.3]{CVW20Linkage}, the ideal $I$ is linked to an ideal $I'$ of format $(5,s-3)$ with $p=0$ and $q\geq 1$. Since $p=0$, the ideal $I'$ must be of class $\mathbf{G}(r)$ or $\mathbf{H}(0,q)$, by \ref{classification}. Since $m=5$, if $I'$ is of class $\mathbf{G}(r)$, then we must have $s-3=1$, by \cite[Theorem 4.5(b)]{CVW20Linkage}. However, $s-3\geq3$, so $I'$ must be of class $\mathbf{H}(0,q)$. By \Cref{rmk:Hoptions}.\ref{eq:qMax}, we have that $q\leq s-3$. However, by \Cref{rmk:Hoptions}.\ref{eq:pqMax}, $p\neq 5-1$ implies $q\neq s-3$. Therefore by \Cref{rmk:Hoptions}.\ref{eq:qMax2}, we have that $q\leq s-4= 5-4=1$, but since $q\geq 1$, we must have $q=1$. Hence $I'$ is of class $\mathbf{H}(0,1)$ and we have realized the family of classes \ref{m=5}$(0)$.

Let $J$ denote an ideal of class $\mathbf{H}(1,2)$ and format $(s,2)$ for $s\geq 6$ even, listed in \Cref{lem:RealizeType2}(i). By \Cref{linkH}(iii), $J$ is linked to an ideal of class $\mathbf{H}(1,1)$ and format $(5,s-2)$, realizing the family \ref{m=5}$(1)$. By \Cref{linkH}(i), $J$ is linked to an ideal of class $\mathbf{H}(2,1)$ and format $(5,s-1)$. Noticing that the ideal of class $\mathbf{H}(2,1)$ and format $(5,3)$ was realized in \ref{n=3}$(1)$ above, this fully realizes the family \ref{m=5}$(2)$. By \Cref{linkH}(iv), $J$ is linked to an ideal of class $\mathbf{H}(3,1)$ and format $(5,s-2)$, realizing the family \ref{m=5}$(3)$. By \Cref{linkH}(ii), $J$ is linked to an ideal of class $\mathbf{H}(4,1)$ and format $(5,s-1)$, realizing the family \ref{m=5}$(4)$. 

Now, we proceed with a double induction argument on $m$ and $n$, broken into two parts with two steps each. Let $k$ be a positive integer. 
In the first part, we show that the realizability of boundary classes for $n=2k+1$ implies the realizability of boundary classes for $m=2k+4$. Then, we show that the realizability of boundary classes for $m=2k+4$ implies the realizability of classes for $n=2(k+1)+1=2k+3$. The base case for this part occurs when $n=3$, which was completed above.
In the second part, we show that the realizability of  classes for $m=2k+3$ implies the realizability of boundary classes for $n=2k+2$. Then, we show that the realizability of boundary classes for $n=2k+2$ implies the realizability of boundary classes for $m=2(k+1)+3=2k+5$. The base case for this part occurs when $m=5$, which was completed above.
Hence the first part realizes all boundary classes for formats $(m,n)$ when $m$ is even or $n$ is odd, and the second part realizes all boundary classes for formats $(m,n)$ when $m$ is odd or $n$ is even, covering all cases.

\newpage
\noindent\underline{Part 1a: realizability for $n=2k+1$ implies realizability for $m=2k+4$}

Assume that we can realize all permissible boundary classes for formats $(m,2k+1)$. That is, by \Cref{rmk:Hoptions}, assume we can realize classes $\mathbf{H}(2k,q)$ for formats $(m,2k+1)$ when $m\geq 2k+1$, $q\equiv_2 m-4$ and $q\leq \min\{m-4,2k+1\}$. We denote this family of classes as follows for $0\leq q\leq 2k+1$:
\addtocounter{Hproof}{1}
\begin{center}
\begin{tabular}{ccccccccc}
 &\quad& class &\quad& \textnormal{format} &\quad& \textnormal{parameters}\\
\newref{\theHproof}$(q)$\label{n=2k+1}  &\quad& $\mathbf{H}(2k,q)$  &\quad& $(m,2k+1)$ &\quad& $m\geq\max\{q+4,2k+1\}$, $q\equiv_2 m-4$ 
\end{tabular}
\end{center}

By \Cref{rmk:Hoptions}, the permissible boundary classes for formats $(2k+4,n)$ are $\mathbf{H}(p,2k)$ with $n\geq 2k$, $p\equiv_2 n-1$, and $p\leq \min\{n-1,2k+3\}$. We denote this family of classes as follows for $0\leq p\leq 2k+3$:
\addtocounter{Hproof}{1}
\begin{center}
\begin{tabular}{ccccccccc}
 &\quad& class &\quad& \textnormal{format} &\quad& \textnormal{parameters}\\
\newref{\theHproof}$(p)$\label{m=2k+4}  &\quad& $\mathbf{H}(p,2k)$  &\quad& $(2k+4,n)$  &\quad& $n\geq\max\{p+1,2k\}$, $p\equiv_2 n-1$ 
\end{tabular}
\end{center}

To realize these families, we will use three different results from \Cref{linkH}: one to realize the family \ref{m=2k+4}$(0)$, one to realize the family \ref{m=2k+4}$(1)$, and one to realize the families \ref{m=2k+4}$(p)$ for $2\leq p\leq n-1$.

Let $I$ denote an ideal of class $\mathbf{H}(2k,0)$ and format $(s,2k+1)$ with $s\geq 2k+3$ even, listed in \ref{n=2k+1}$(0)$. Since $k\geq1$, this ensures $2\leq 2k\leq s-3$ and therefore by \Cref{linkH}(v), $I$ is linked to an ideal of class $\mathbf{H}(0,2k)$ and format $(2k+4,s-3)$, realizing the family \ref{m=2k+4}$(0)$. 
By \Cref{linkH}(iv), $I$ is also linked to an ideal of class $\mathbf{H}(1,2k)$ and format $(2k+4,s-2)$, realizing the family \ref{m=2k+4}$(1)$.

Now, we realize the families \ref{m=2k+4}$(p)$ for $2\leq p\leq n-1$. Recall that the ideal of class $\mathbf{H}(2,2k)$ and format $(2k+4,3)$ in family \ref{m=2k+4}$(2)$ has been realized in the base case, so we realize this family for $n\geq 5$ odd. Let $J$ denote an ideal of class $\mathbf{H}(2k,i)$ and format $(s,2k+1)$ for some fixed $i\in\{0,1,2,3,\dots,2k+1\}$, listed in \ref{n=2k+1}$(i)$. By \Cref{linkH}(ii), $J$ is linked to an ideal $J'$ of class $\mathbf{H}(i+2,2k)$ and format $(2k+4,s-1)$. It is straightforward to verify that for all $i\in\{0,1,2,3,\dots,2k+1\}$, the parameters on $s$ for ideal $J$ translate to the corresponding parameters on $n$ for $J'$ to realize the classes \ref{m=2k+4}$(p)$ for $2\leq p\leq n-1$. This completes the first argument.\\

\noindent\underline{Part 1b: realizability for $m=2k+4$ implies realizability for $n=2k+3$}

Assume that we can realize all permissible boundary classes for formats $(2k+4,t)$, listed in \ref{m=2k+4}$(p)$ above. By \Cref{rmk:Hoptions}, the permissible boundary classes for formats $(m,2k+3)$ are $\mathbf{H}(2k+2,q)$ with $m\geq 2k+3$, $q\equiv_2 m-4$ and $q\leq \min\{m-4,2k+3\}$. We denote this family of classes as follows for $0\leq q\leq 2k+3$:
\addtocounter{Hproof}{1}
\begin{center}
\begin{tabular}{ccccccccc}
 &\quad& class &\quad& \textnormal{format} &\quad& \textnormal{parameters}\\
\newref{\theHproof}$(q)$\label{n=2k+3}  &\quad& $\mathbf{H}(2k+2,q)$  &\quad& $(m,2k+3)$ &\quad& $m\geq\max\{q+4,2k+3\}$, $q\equiv_2 m-4$ 
\end{tabular}
\end{center}

Let $I$ denote an ideal of class $\mathbf{H}(i,2k)$ and format $(2k+4,t)$ for some fixed $i\in\{0,1,2,\dots,2k+3\}$, listed in \ref{m=2k+4}$(i)$. \Cref{linkH}(ii), $I$ is linked to an ideal $I'$ of class $\mathbf{H}(2k+2,i)$ and format $(t+3,2k+3)$. It is straightforward to verify that for all $i\in\{0,1,2,\dots,2k+3\}$, the parameters on $t$ for ideal $I$ translate to the corresponding parameters on $m$ for $I'$ to realize the families of classes listed in \ref{n=2k+3}$(q)$ for $0\leq q\leq 2k+3$. This completes the second argument.\\

\noindent\underline{Part 2a: realizability for $m=2k+3$ implies realizability for $n=2k+2$}

Assume that we can realize all permissible boundary classes for formats $(2k+3,t)$. That is, by \Cref{rmk:Hoptions}, assume we can realize classes $\mathbf{H}(p,2k-1)$ for formats $(2k+3,t)$ when $p\equiv_2 t-1$, $p\leq \min\{t-1,2k+2\}$, and $t\geq 2k-1$. We denote this family of classes as follows for $0\leq p\leq 2k+2$:
\addtocounter{Hproof}{1}
\begin{center}
\begin{tabular}{ccccccccc}
 &\quad& class &\quad& \textnormal{format} &\quad& \textnormal{parameters}\\
\newref{\theHproof}$(p)$\label{m=2k+3}  &\quad& $\mathbf{H}(p,2k-1)$  &\quad& $(2k+3,n)$ &\quad& $n\geq\max\{p+1,2k-1\}$, $p\equiv_2 n-1$  
\end{tabular}
\end{center}

By \Cref{rmk:Hoptions}, the permissible boundary classes for formats $(m,2k+2)$ are $\mathbf{H}(2k+1,q)$ with $m\geq 2k+2$, $q\equiv_2 m-4$ and $q\leq \min\{m-4,2k+2\}$. We denote this family of classes as follows for $0\leq q\leq 2k+2$:
\addtocounter{Hproof}{1}
\begin{center}
\begin{tabular}{ccccccccc}
 &\quad& class &\quad& \textnormal{format} &\quad& \textnormal{parameters}\\
\newref{\theHproof}$(q)$\label{n=2k+2}  &\quad& $\mathbf{H}(2k+1,q)$  &\quad& $(m,2k+2)$  &\quad& $m\geq\max\{q+4,2k+2\}$, $q\equiv_2 m-4$ 
\end{tabular}
\end{center}

Let $I$ denote an ideal of class $\mathbf{H}(i,2k-1)$ and format $(2k+3,t)$ for some fixed $i\in\{0,1,2,\dots,2k+2\}$, listed in \ref{m=2k+3}$(i)$. By \Cref{linkH}(ii), $I$ is linked to an ideal $I'$ of class $\mathbf{H}(2k+1,i)$ and format $(t+3,2k+2)$. It is straightforward to verify that for all $i\in\{0,1,2,\dots,2k+2\}$, the restrictions on $t$ for ideal $I$ translate to the corresponding restriction on $m$ for $I'$ to realize the families of classes listed in \ref{n=2k+2}$(q)$. This completes the third argument.\\

\noindent\underline{Part 2b: realizability for $n=2k+2$ implies realizability for $m=2k+5$}

Assume that we can realize all permissible boundary classes for formats $(s,2k+2)$, listed in \ref{n=2k+2}$(q)$ above. By \Cref{rmk:Hoptions}, the permissible boundary classes for formats $(2k+5,n)$ are $\mathbf{H}(p,2k+1)$ with $n\geq 2k+1$, $p\equiv_2 n-1$ and $p\leq \min\{n-1,2k+4\}$. We denote this family of classes as follows for $0\leq p\leq 2k+4$:
\addtocounter{Hproof}{1}
\begin{center}
\begin{tabular}{ccccccccc}
 &\quad& class &\quad& \textnormal{format} &\quad& \textnormal{parameters}\\
\newref{\theHproof}$(p)$\label{m=2k+5}  &\quad& $\mathbf{H}(p,2k+1)$  &\quad& $(2k+5,n)$  &\quad& $n\geq\max\{p+1,2k+1\}$, $p\equiv_2 n-1$ 
\end{tabular}
\end{center}

To realize these families, we will use three different results from \Cref{linkH}: one to realize the family \ref{m=2k+5}$(0)$, one to realize the family \ref{m=2k+5}$(1)$, and one to realize the families \ref{m=2k+5}$(p)$ for $2\leq p\leq n-1$.

Let $I$ denote an ideal of class $\mathbf{H}(2k+1,0)$ and format $(s,2k+2)$ with $s\geq \max\{6,2k+4\}$ even, listed in \ref{n=2k+2}$(0)$. Since $k\geq1$, this ensures $2\leq 2k+1\leq m-3$ and therefore by \Cref{linkH}(v), $I$ is linked to an ideal $I'$ of class $\mathbf{H}(0,2k+1)$ and format $(2k+5,m-3)$, realizing the family \ref{m=2k+5}$(0)$. By \Cref{linkH}(iv), $I$ is also linked to an ideal of class $\mathbf{H}(1,2k+1)$ and format $(2k+3,m-2)$, realizing the family \ref{m=2k+5}$(1)$.

Let $J$ denote an ideal of class $\mathbf{H}(2k+1,i)$ and format $(s,2k+2)$ for some fixed $i\in\{0,1,2,3,\dots,2k+2\}$, listed in \ref{n=2k+2}$(i)$. By \Cref{linkH}(ii), $J$ is linked to an ideal $J'$ of class $\mathbf{H}(i+2,2k+1)$ and format $(2k+5,s-1)$. It is straightforward to verify that for all $i\in\{0,1,2,3,\dots,2k+2\}$, the restrictions on $s$ for ideal $J$ translate to the corresponding restriction on $n$ for $J'$ to realize the classes \ref{m=2k+5}$(p)$ for $2\leq p\leq 2k+4$. This completes the fourth argument.
\end{proof}

In the cases $m=4$ or $n=2$, the permissible boundary classes differ than those listed in the statement of \Cref{thm:realizeH}, but we can also realize all permissible boundary classes.

\begin{proposition}\label{prop:Hm4n2}
For formats $(4,n)$ and $(m,2)$, one can realize all permissible boundary classes. 
\end{proposition}

\begin{proof}
In the case $m=4$, the permissible boundary classes are 
\begin{align*}
\mathbf{H}(p,0), \text{ with } p\leq\min\{3,n-1\}  \text{ and } p\equiv_2 n-1.
\end{align*}
However, by \cite[3.4.2]{Avramov12} the only permissible $\mathbf{H}$ classes when $m=4$ are $\mathbf{H}(3,2)$ when $n=2$, which is a hypersurface section, and $\mathbf{H}(3,0)$ when $n\geq 3$. Therefore, the only permissible boundary classes are 
\begin{align*}
\mathbf{H}(3,0), \text{ with } n\geq 4 \text{ even},
\end{align*}
which are all realized in \Cref{lem:RealizeACI}.

In the case $n=2$, the permissible boundary classes are 
\begin{align*}
&\mathbf{H}(1,q), \text{ with } q\leq\min\{m-4,2\}  \text{ and } q\equiv_2 m-4
\end{align*}
However, by \cite[Theorem 4.4]{Brown84} the only permissible $\mathbf{H}$ classes when $n=2$ are $\mathbf{H}(3,2)$ when $m=4$, which is a hypersurface section, and $\mathbf{H}(1,2)$ when $m\geq 5$. Therefore, the only permissible boundary classes are 
\begin{align*}
\mathbf{H}(1,2), \text{ with } m\geq 6 \text{ even},
\end{align*}
which are all realized in \Cref{lem:RealizeType2}. Therefore, one can realize all permissible boundary classes when $m=4$ and $n=2$.
\end{proof}

\section*{Acknowledgements}
The author would like to thank her advisor Lars W. Christensen for his guidance and support on this project as well as the anonymous referee for a very close reading and numerous helpful comments that greatly improved the presentation.

\appendix
\section{Multiplication in Tor Algebras of Linked Ideals---Proofs}\label{appendix}

\begin{setup}\label{setup:DefLink}
This notation, and the notation throughout this appendix, largely follows the notation found in \cite{AKM88}. Let $(R,\mathfrak{m},\Bbbk)$ be a regular local ring and $I\subseteq\mathfrak{m}^2$ a perfect ideal of grade 3. Consider a minimal free resolution $F_{\bullet}$ of $R/I$ as 
\begin{equation*}
\begin{tikzcd} 
0 \ar[r] & F_{3} \ar[r,"d_3"] & F_{2} \ar[r,"d_2"] & F_{1} \ar[r,"d_1"] & R \ar[r] & 0.
\end{tikzcd}
\end{equation*}
Set $m = \rank_R(F_1)$, set $n = \rank_R(F_3)$ and denote the tuple $(m,n)$ as the \emph{format} of $I$. Basis elements of $F_1$ will be denoted by $e$, basis elements of $F_2$ will be denoted by $f$, and basis elements of $F_3$ will be denoted by $g$. 

Let $\mathbf{x} = x_1,x_2,x_3$ be a regular sequence in $I$ that does not generate $I$. Consider the linked ideal $J = (\mathbf{x}):I$ and the Koszul resolution $(K_\bullet,\delta)$ of $R/(\mathbf{x})$. Let $\mathbf{\phi}:K_\bullet \rightarrow F_\bullet$ be the morphism of DG-algebras that extends the natural map $R/(\mathbf{x}) \rightarrow R/I$ and consider the commutative diagram:
\begin{equation}\label{a:commdiag}
\begin{tikzcd}
0 \ar[r]
& K_3 \ar[r, "\delta_3"] \ar[d, "\phi_3"] 
& K_2 \ar[r, "\delta_2"] \ar[d, "\phi_2"]  
& K_1 \ar[r, "\delta_1"] \ar[d, "\phi_1"]
& R \ar[r] \ar[d, equals, "\phi_0"]
& R/(\mathbf{x}) \ar[r] \ar[d]
& 0\\
0 \ar[r]
& F_3 \ar[r, "d_3"]
& F_2 \ar[r, "d_2"] 
& F_1 \ar[r, "d_1"]
& R \ar[r] 
& R/I \ar[r]
& 0
\end{tikzcd}    
\end{equation}
Fix bases 
\[ 
u_1,u_2,u_3, \quad v_{1,2},v_{2,3},v_{1,3}, \quad\textnormal{and}\quad w
\]
for $K_1$, $K_2$, and $K_3$, respectively, such that $K_\bullet$ has a DG-algebra structure given by
\begin{align*}
u_1u_2 = v_{1,2}, \quad u_1u_3&=v_{1,3}, \quad u_2u_3=v_{2,3}\\
u_1u_2u_3 &= w
\end{align*}
where the only nontrivial products that are not listed are given by the rules of graded-commutativity. For any $R$-module $M$, let $M^*$ denote the linear dual $\Hom_R(M,R)$. We consider $K^*$ to have dual bases 
\[ 
u_1^*,u_2^*,u_3^* \quad\quad v_{1,2}^*,v_{2,3}^*,v_{1,3}^* \quad\quad w^*
\]
for $K_1^*$, $K_2^*$, and $K_3^*$, respectively. The product on $K_\bullet$ induces perfect pairings $K_i \otimes K_{3-i} \rightarrow K_3$ for $0\leq i\leq 3$. 
The perfect pairings induce isomorphisms $\chi_i: K_i^*\rightarrow K_{3-i}$ via
\begin{align*}
\chi_0(1^*)&=w & \chi_1(u_1^*)&=v_{2,3} & \chi_2(v_{2,3}^*)&=u_1 & \chi_3(w^*)&=1\\
& & \chi_1(u_2^*)&=v_{1,3} & \chi_2(v_{1,3}^*)&=u_2 & &\\
& & \chi_1(u_3^*)&=v_{1,2} & \chi_2(v_{1,2}^*)&=u_3 & &
\end{align*}
For $0\leq i \leq 3$, define $\psi_i = \chi_i \circ \phi_i^*: F_i^* \rightarrow K_i^*\rightarrow K_{3-i}$. 
\end{setup}

\noindent We now recall a result from \cite[Proposition 1.6]{AKM88} on a minimal free resolution of $R/J$:

\begin{chunk}
Adopt the notation and hypotheses from \Cref{setup:DefLink}. The complex $\textnormal{cone}(\psi)$ gives a resolution of $R/J$. After splitting off a contractible summand defined by the isomorphism $\psi_0$, we denote the resulting complex $(D_\bullet,\partial)$. This complex is given by
\begin{equation*}
\begin{tikzcd}[row sep = 0.000001em]
& 
& F_2^*
& F_3^*
\\
0 \ar[r]
& F_1^* \ar[r, "\partial_3"] 
& \oplus \ar[r, "\partial_2"] 
& \oplus \ar[r, "\partial_1"]
& K_0 \ar[r]
& 0
\\
& 
& K_2
& K_1
\end{tikzcd}
\end{equation*}
where the generators of $D_1=F_3^*\oplus K_1$ are 
\[
\begin{bmatrix}
g_i^* \\ 0
\end{bmatrix} \textnormal{ for } 1\leq i\leq n,\enspace 
\begin{bmatrix}
0 \\ u_1
\end{bmatrix}, \enspace 
\begin{bmatrix}
0 \\ u_2
\end{bmatrix},\enspace\textnormal{and}\enspace
\begin{bmatrix}
0 \\ u_3
\end{bmatrix},
\]
the generators of $D_2=F_2^*\oplus K_2$ are
\[
\begin{bmatrix}
f_i^* \\ 0
\end{bmatrix} \textnormal{ for } 1\leq i\leq m+n-1,\enspace 
\begin{bmatrix}
0 \\ v_{2,3}
\end{bmatrix},\enspace
\begin{bmatrix}
0 \\ v_{1,3}
\end{bmatrix}, \enspace\textnormal{and}\enspace
\begin{bmatrix}
0 \\ v_{1,2}
\end{bmatrix},
\]
and the generators of $D_3=F_1^*$ are
\[
e_i^* \textnormal{ for } 1\leq i\leq m.
\]
By construction, the differentials are given by the matrices 
\begin{equation*}
\partial_3 = \begin{bmatrix}
d_2^*\\
\psi_1
\end{bmatrix}, \quad
\partial_2 = \begin{bmatrix}
d_3^* & 0\\
-\psi_2 & -\delta_2
\end{bmatrix}, \enspace \text{ and } \enspace
\partial_1 = \begin{bmatrix}
\psi_3 & \delta_1
\end{bmatrix}.
\end{equation*}
Moreover, if $F_\bullet$ is a minimal resolution of $R/I$, then 
\begin{equation}\label{a:ranks}
\rank(\psi_i\otimes \Bbbk) = \rank(\phi_i\otimes \Bbbk) = \rank(\partial_{3-i}\otimes \Bbbk) \text{ for all } 1\leq i\leq 3.  
\end{equation}
We now look at the different options for the format of $J$ depending on $\rank(\phi_i\otimes \Bbbk)$ for all $1\leq i\leq 3$.
\end{chunk}

\begin{lemma}\label{LinkOption}
Adopt the notation and hypotheses from \Cref{setup:DefLink}. Let $I$ be an ideal with format $(m,n)$ and let $J=I:(x_1,x_2,x_3)$ be the ideal obtained from $I$ via linking over a regular sequence $x_1,x_2,x_3$. Then $J$ is an ideal with format given from the following list: $(n+3,m)$, $(n+3,m-1)$, $(n+3,m-2)$, $(n+2,m-2)$, $(n+3,m-3)$, $(n+2,m-3)$, $(n+1,m-3)$, $(n,m-3)$.
\end{lemma}

\begin{proof} 
The Tor algebra $A_\bullet=\Tor_\bullet^R(R/I,\Bbbk)$ has the structure of a graded algebra induced from the differential graded algebra structure on $F_\bullet$. Within our assumptions, $A_\bullet$ has a multiplicative structures given by one of the classes listed in \Cref{classification}. Of these structures, only class $\mathbf{C}(3)$ has nontrivial multiplication of three elements from $A_1$. Since $\phi$ is an algebra homomorphism, for $e,e',e''\in F_1$ and $u,u',u''\in K_1$ such that $\phi_1(u)=e$, $\phi_1(u')=e'$, and $\phi_1(u'')=e''$, we have that $\phi_3(u\cdot u'\cdot u'') = \phi_1(u) \cdot \phi_1(u') \cdot \phi_1(u'') = e\cdot e'\cdot e''$ and thus $\rank(\phi_3\otimes \Bbbk)=0$ for all classes except $\mathbf{C}(3)$. 
Since $I$ is not a complete intersection ideal, $I$ will not be of class $\mathbf{C}(3)$. Noticing that $\rank(\phi_3\otimes \Bbbk)=0$, we obtain from the equalities in \eqref{a:ranks} that the format of $J$ is given by 
\[
(n+3-\rank(\phi_2\otimes \Bbbk), m-\rank(\phi_1\otimes \Bbbk)).
\]

Depending on the choice of regular sequence, we have $\rank(\phi_1\otimes \Bbbk)\in\{0,1,2,3\}$. We consider each of these cases and use the fact that $\phi$ is an algebra homomorphism to derive information about $\rank(\phi_2\otimes \Bbbk)$.

If $\rank(\phi_1\otimes \Bbbk)=0$, then $\rank(\phi_2\otimes \Bbbk)=0$. This accounts for the format $(n+3,m)$. 

If $\rank(\phi_1\otimes \Bbbk)=1$, then $\rank(\phi_2\otimes \Bbbk)=0$ by skew commutativity. This accounts for format $(n+3,m-1)$.

If $\rank(\phi_1\otimes \Bbbk)=2$, then by skew-commutativity , we have $\rank(\phi_2\otimes \Bbbk)=0$ or $\rank(\phi_2\otimes \Bbbk)=1$ depending on if the product $e_1e_2$ has coefficients in $\mathfrak{m}$. This accounts for the formats $(n+3,m-2)$ and $(n+2,m-2)$. 

If $\rank(\phi_1\otimes \Bbbk)=3$, then by skew-commutativity, we consider the number of pairwise nontrivial products of basis elements, of which there can be $0$, $1$, $2$, or $3$. This implies $\rank(\phi_2\otimes \Bbbk)\in\{0,1,2,3\}$, accounting for the final four formats listed in the statement.
\end{proof}

Notice that each format arises from distinct values of $\rank(\phi_1\otimes \Bbbk)$ and $\rank(\phi_2\otimes \Bbbk)$, described in \Cref{tab:4} below. In this work, we focus on only the first five possibilities. The remaining possibilities are discussed in works such as \cite{CVW20Linkage}.

\begin{table}[h]
    \centering
    \begin{tabular}{|c|c|c|c|c|}
    \hline
    & $\rank(\phi_2\otimes\Bbbk)=0$ &  $\rank(\phi_2\otimes\Bbbk)=1$ & $\rank(\phi_2\otimes\Bbbk)=2$ & $\rank(\phi_2\otimes\Bbbk)=3$ \\
    \hline
    $\rank(\phi_1\otimes\Bbbk)=0$ & $(n+3,m)$ & -- & -- & --\\
    \hline
    $\rank(\phi_1\otimes\Bbbk)=1$ & $(n+3,m-1)$ & -- & -- & --\\
    \hline
    $\rank(\phi_1\otimes\Bbbk)=2$ & $(n+3,m-2)$ & $(n+2,m-2)$ & -- & --\\
    \hline
    $\rank(\phi_1\otimes\Bbbk)=3$ & $(n+3,m-3)$ & $(n+2,m-3)$ & $(n+1,m-3)$ & $(n,m-3)$ \\
    \hline
    \end{tabular}
    \caption{Possible Formats of a Linked Ideal}
    \label{tab:4}
\end{table}

Now, we state results from the literature \cite[Lemma 1.9, Lemma 1.10, Theorem 1.13]{AKM88} that detail how the product structure on $\Tor_\bullet^R(R/I,\Bbbk)$ impacts the product structure on $\Tor_\bullet^R(R/J,\Bbbk)$. 

\begin{chunk}
Adopt the notation and hypotheses from \Cref{setup:DefLink}. Let $\langle \cdot,\cdot \rangle:M\times M^*\rightarrow R$ denote the evaluation morphism for any $R$-module $M$. By \cite[Lemma 1.9]{AKM88}, there exists a map 
\[
X:\bigwedge^3 F_1 \otimes \bigwedge^2 F_3^* \rightarrow F_2^*
\]
with 
\[
d_3^* X(e\wedge e' \wedge e'' \otimes g^* \wedge g'^*) = \langle ee'e'', g^*\rangle g'^* - \langle ee'e'', g'^*\rangle g^*
\]
for all $e,e',e''\in F_1$ and all $g^*,g'^*\in F_3^*$. Moreover, by \cite[Lemma 1.10]{AKM88} there exists a map 
\[
Y:\bigwedge^3 F_1 \otimes F_3^* \otimes F_2^*  \rightarrow F_1^*
\]
with $\langle f', d_2^*\circ Y(e \wedge e' \wedge e'' \otimes g^* \otimes f^*)\rangle$ equal to 
\begin{align*}
\langle e e' e'', g^* \rangle \langle f', f^* \rangle &- \langle f', X(e \wedge e' \wedge e'' \otimes g^* \otimes d_3^*f^*) \rangle\\
- \langle e' e'', f^* \rangle \langle ef', g^*\rangle &+ \langle e e'', f^* \rangle \langle e'f', g^*\rangle - \langle e e', f^* \rangle \langle e''f', g^*\rangle
\end{align*}
for all $e,e',e''\in F_1$, $f'\in F_2$, $f^*\in F_2^*$, and $g^*\in F_3^*$.

Let $\Sigma_{rst}$ denote the sum taken over the indices $r$, $s$, and $t$ with $1\leq r<s<t\leq n$ and let $(\phi_1)_{rst}$ denote the minor of the matrix representation of $\phi_1$ taken over rows $r$, $s$, and $t$. Then by \cite[Theorem 1.13]{AKM88} there exist maps $\lambda:\bigwedge^2 F_3^*\rightarrow K_2$ and $\mu:F_3^*\otimes F_2^*\rightarrow R^*$ such that the following multiplication endows $D_\bullet$ with the structure of a differential graded algebra. The products $D_1\times D_1 \rightarrow D_2$ are given by
\begin{enumerate}[label=\alph*.]
    \item $\begin{bmatrix}
    0 \\ u
    \end{bmatrix} \cdot \begin{bmatrix}
    0 \\ u'
    \end{bmatrix} = 
    \begin{bmatrix}
    0 \\ u u'
    \end{bmatrix}$
    \item $\begin{bmatrix}
    0 \\ u
    \end{bmatrix} \cdot \begin{bmatrix}
    g^* \\ 0
    \end{bmatrix} = 
    \begin{bmatrix}
    \sum_{k=1}^{m+n-1} \langle \phi_1(u)f_k, g^* \rangle f_k^* \\ 0
    \end{bmatrix}$
    \item $\begin{bmatrix}
    g^* \\ 0
    \end{bmatrix} \cdot \begin{bmatrix}
    g'^* \\ 0
    \end{bmatrix} = 
    \begin{bmatrix}
    \sum_{rst} (\phi_1)_{rst} X(e_r\wedge e_s \wedge e_t \otimes g^* \wedge g'^*) \\ \lambda(g^* \wedge g'^*)
    \end{bmatrix}$
\end{enumerate}
and the products $D_1\times D_2 \rightarrow D_3$ are given by 
\begin{enumerate}[label=\alph*.]
    \setcounter{enumi}{3}
    \item $\begin{bmatrix}
    0 \\ u
    \end{bmatrix} \cdot \begin{bmatrix}
    0 \\ v
    \end{bmatrix} = \langle uv,w^*\rangle d_1$
    \item $\begin{bmatrix}
    0 \\ u
    \end{bmatrix} \cdot \begin{bmatrix}
    f^* \\ 0
    \end{bmatrix} = 
    \sum_{k=1}^m \langle \phi_1(u)e_k, f^* \rangle e_k^* $
    \item $\begin{bmatrix}
    g^* \\ 0
    \end{bmatrix} \cdot \begin{bmatrix}
    0 \\ v
    \end{bmatrix} = 
    \sum_{k=1}^m \langle \phi_2(v)e_k, g^* \rangle e_k^* $
    \item $\begin{bmatrix}
    g^* \\ 0
    \end{bmatrix} \cdot \begin{bmatrix}
    f^* \\ 0
    \end{bmatrix} = 
    \sum_{rst} (\phi_1)_{rst} Y(e_r\wedge e_s \wedge e_t \otimes g^* \wedge f^*) + d_1^*(\mu(g^*\otimes f^*))$
\end{enumerate}
Apart from the products given by graded-commutativity, the products not listed are zero. Moreover, one may assume that $\im\lambda \subseteq \mathfrak{a}_1\mathfrak{a}_2\mathfrak{a}_3K_2$, where $\mathfrak{a}_i$ is the ideal generated by the entries of the $i^{th}$ column of $\phi_1$.
\end{chunk}

\noindent In the remainder of this appendix, it will be useful to set the following notation.

\begin{notation}\label{notation:JInfo}
Adopt the notation and hypotheses from \Cref{setup:DefLink}. We denote by $C_\bullet$ a minimal free resolution of linked ideal $J$ obtained from $D_\bullet$. We denote by $B_\bullet$ the Tor algebra $\Tor_\bullet^R(R/J,\Bbbk)$. Let $\overline{\cdot}=\cdot \otimes \Bbbk$. We denote by
\[
\EE_i = \overline{\begin{bmatrix}
g_i^* \\ 0
\end{bmatrix}} \textnormal{ for } 1\leq i\leq n,\enspace 
\EE_{n+1} = \overline{\begin{bmatrix}
0 \\ u_1
\end{bmatrix}}, \enspace 
\EE_{n+2} = \overline{\begin{bmatrix}
0 \\ u_2
\end{bmatrix}},\enspace\textnormal{and}\enspace
\EE_{n+3} = \overline{\begin{bmatrix}
0 \\ u_3
\end{bmatrix}},
\]
\[
\FF_i = \overline{\begin{bmatrix}
f_i^* \\ 0
\end{bmatrix}} \textnormal{ for } 1\leq i\leq m+n-1,\enspace 
\FF_{m+n} = \overline{\begin{bmatrix}
0 \\ v_{2,3}
\end{bmatrix}},\enspace
\FF_{m+n+1} = \overline{\begin{bmatrix}
0 \\ v_{1,3}
\end{bmatrix}}, \enspace\textnormal{and}\enspace
\FF_{m+n+2} = \overline{\begin{bmatrix}
0 \\ v_{1,2}
\end{bmatrix}},
\]
\[
\GG_i = \overline{e_i^*} \textnormal{ for } 1\leq i\leq m,
\]
the basis elements of $B_1$, $B_2$, and $B_3$, respectively, induced by the basis on $C_\bullet$. Note that it is not guaranteed that all of these elements will be present once we consider the \emph{minimal} resolution. This will be addressed on a case by case basis in \Cref{lem:up6}-\ref{lem:over3}.
\end{notation}

\begin{setup}\label{setup:LinkClass}
Adopt the notation and hypotheses as in \Cref{notation:JInfo}. If $\rank(\phi_1\otimes \Bbbk)=t$ for some $t\geq 1$, then we will assume we have $x_1,\dots x_t$ as minimal generators of $I$ that correspond to basis elements $e_1,\dots,e_t$ of $F_1$, i.e., $d_1(e_i)=x_i$ for all $1\leq i\leq t$. By \cite[Lemma 8.2]{BE74}, we can still assume $x_1,x_2,x_3$ forms a regular sequence in a way that does not impact the multiplicative structure of $\Tor_\bullet^R(R/I,\Bbbk)$, see \cite[A.5]{CVW20Linkage}.

By \eqref{a:commdiag}, since $d_1\circ \phi_1=\delta_1$, we also have that $\phi_1(u_i)=e_i$ for all $1\leq i\leq t$. Moreover, if $\rank(\phi_2\otimes \Bbbk)=1$, we will assume that $e_1\cdot e_2=f_1$. Since $\phi$ is a DG-algebra morphism, we also have that $\phi_2(v_{1,2})=\phi_2(u_1\cdot u_2)=\phi_1(u_1)\cdot \phi_1(u_2)=e_1\cdot e_2=f_1$.
\end{setup}

\noindent The following result can be deduced from the proof of \cite[Lemma 2.8]{AKM88}, but we include it here for completeness.

\begin{lemma}\label{lem:up6}
Adopt the notation and hypotheses as in \Cref{setup:LinkClass}.
Consider an ideal $J$, linked to $I$, such that $\rank(\phi_1\otimes k)=0$. Then $B_\bullet$ has only the nontrivial products $\EE_{n+1}\cdot \EE_{n+2} = \FF_{m+n+2}$, $\EE_{n+1}\cdot \EE_{n+3} = \FF_{m+n+1}$, and $\EE_{n+2}\cdot \EE_{n+3} = \FF_{m+n+2}$. In particular, $J$ is of class $\mathbf{T}$. 
\end{lemma}

\begin{proof}
Assume $\rank(\phi_1\otimes k)=0$. Then $\phi_1(u_1), \phi_1(u_2), \phi_1(u_3) \in\mathfrak{m}F_1$ and by the proof of \Cref{LinkOption}, we have that $\rank(\phi_2\otimes k)=0$.  This implies products b, c, e, and g are zero in $B_\bullet$. Moreover, $\phi_2(v_{1,2}), \phi_2(v_{1,3}), \phi_2(v_{2,3}) \in\mathfrak{m}F_2$, so product f is zero in $B_\bullet$. Finally, $d_1$ takes image in $\mathfrak{m}$, so product d is zero in $B_\bullet$. Therefore we only have nontrivial products coming from product a, which are $\EE_{n+1}\cdot \EE_{n+2} = \FF_{m+n+2}$, $\EE_{n+1}\cdot \EE_{n+3} = \FF_{m+n+1}$, and $\EE_{n+2}\cdot \EE_{n+3} = \FF_{m+n+2}$. Hence $J$ is of class $\mathbf{T}$.
\end{proof}

\begin{lemma}\label{lem:up4}
Adopt the notation and hypotheses as in \Cref{setup:LinkClass}.
Consider an ideal $J$, linked to $I$, such that $\rank(\phi_1\otimes k)=1$. Then $B_\bullet$ has guaranteed nontrivial products $\EE_{n+1}\cdot \EE_{n+2} = \FF_{m+n+2}$ and $\EE_{n+1}\cdot \EE_{n+3} = \FF_{m+n+1}$, as well as possible nontrivial products 
\begin{align*}
    \EE_{n+1}\cdot \EE_i &= \sum_{k=1}^{m+n-1} \overline{\langle e_1f_k, g_i^*\rangle} \FF_k \textnormal{ for } 1\leq i\leq n,\textnormal{ and}\\ 
    \EE_{n+1}\cdot \FF_i &= \sum_{k=2}^{m} \overline{\langle e_1e_k, f_i^*\rangle} \GG_k \textnormal{ for } 1\leq i\leq m+n-1.
    \end{align*}
\end{lemma}

\begin{proof}
Since $\rank(\phi_1\otimes k)=1$, we have that $\phi_1(u_1) = e_1$ and $\phi_1(u_2), \phi_1(u_3) \in\mathfrak{m}F_1$. To obtain a minimal resolution $C_\bullet$ from $D_\bullet$, we must split off basis elements from $D_3$ and $D_2$. Namely, we split off $G_1$, as this corresponds with the minimal generator $e_1$, and we split off $F_{m+n}$, as this corresponds to $v_{2,3}$ and $\psi_1(e_1^*)=v_{2,3}$. Thus we only consider multiplication between the basis elements $\EE_1,\dots,\EE_{n+3}$, $\FF_1,\dots,\FF_{m+n-1},\FF_{m+n+1},\FF_{m+n+2}$, and $\GG_2,\dots,\GG_m$ in $B_\bullet$. Notice that the absence of $\GG_1$ results in summation index starting at $k=2$ for product e. We consider each product type individually. 
\begin{enumerate}[label=\alph*.]
    \item The product $\EE_{n+2}\cdot \EE_{n+3} = \FF_{m+n}$ has split off and is not in $B_\bullet$. The products $\EE_{n+1}\cdot \EE_{n+2} = \FF_{m+n+2}$ and $\EE_{n+1}\cdot \EE_{n+3} = \FF_{m+n+1}$ are in $B_\bullet$. 
    \item Any products of this type involving $\EE_{n+2}$ and $\EE_{n+3}$ will depend on $\phi_1(u_2)$ and $\phi_1(u_3)$, respectively, which are in $\mathfrak{m}F_1$ and are therefore zero in $B_\bullet$. Therefore we have possible nontrivial products 
    \begin{align*}
        \EE_{n+1}\cdot \EE_i = \sum_{k=1}^{m+n-1} \overline{\langle \phi_1(u_1)f_k, g_i^*\rangle} \FF_k = \sum_{k=1}^{m+n-1} \overline{\langle e_1f_k, g_i^*\rangle} \FF_k \textnormal{ for } 1\leq i\leq n
    \end{align*}
    \item Products of this form involve a sum of terms whose coefficients are $3\times 3$ minors $(\phi_1)_{rst}$ for all $1\leq r<s<t\leq n$. Since $\rank(\phi_1\otimes k)=1$, all minors of this form are in $\mathfrak{m}$. Likewise, since we can take $\im\lambda \subseteq \mathfrak{a}_1\mathfrak{a}_2\mathfrak{a}_3K_2$, where $\mathfrak{a}_j$ is the ideal generated by the entries of the $j^{th}$ column of the matrix representation of $\phi_1$, we have $\im\lambda\subseteq\mathfrak{m}C_\bullet$. Hence there are no products of this form in $B_\bullet$.
    \item The products $\EE_{n+2}\cdot \FF_{m+n+1} = -d_1$ and $\EE_{n+3}\cdot \FF_{m+n+2} = d_1$ are in $\mathfrak{m}F_1^*$ and are therefore zero in $B_\bullet$.
    \item Any products of this type involving $\EE_{n+2}$ or $\EE_{n+3}$ will depend on $\phi_1(u_2)$ or $\phi_1(u_3)$, respectively, which are in $\mathfrak{m}F_1$ and are therefore zero in $B_\bullet$. Therefore we have possible nontrivial products 
    \begin{align*}
        \EE_{n+1}\cdot \FF_i = \sum_{k=2}^{m} \overline{\langle \phi_1(u_1)e_k, f_i^*\rangle} \GG_k = \sum_{k=2}^{m} \overline{\langle e_1e_k, f_i^*\rangle} \GG_k \textnormal{ for } 1\leq i\leq m+n-1
    \end{align*}
    \item The products $\EE_i\cdot \FF_{m+n+1}$ and $\EE_i\cdot \FF_{m+n+2}$ for $1\leq i\leq n$ will depend on $\phi_2(v_{1,3}) = \phi_1(u_1)\cdot \phi_1(u_3)$ and $\phi_2(v_{1,2}) = \phi_1(u_1)\cdot\phi_1(u_2)$, respectively, which are in $\mathfrak{m}F_2$ and are therefore zero in $B_\bullet$.
    \item Products of this form involve a sum of terms whose coefficients are $3\times 3$ minors $(\phi_1)_{rst}$ for all $1\leq r<s<t\leq n$. Since $\rank(\phi_1\otimes k)=1$, all minors of this form are in $\mathfrak{m}$. Likewise, $d_1^*(\mu(g^*\otimes f^*))\in\mathfrak{m}F_1^*$, and all products of this form are zero in $B_\bullet$. \qedhere
\end{enumerate}
\end{proof}

\begin{lemma}\label{lem:up2}
Adopt the notation and hypotheses as in \Cref{setup:LinkClass}.
Consider an ideal $J$, linked to $I$, such that $\rank(\phi_1\otimes k)=2$ and $\rank(\phi_2\otimes k)=0$. Then $B_\bullet$ has guaranteed nontrivial product $\EE_{n+1}\cdot \EE_{n+2} = \FF_{m+n+2}$, as well as possible nontrivial products 
\begin{align*}
    \EE_{n+1}\cdot \EE_i &= \sum_{k=1}^{m+n-1} \overline{\langle e_1f_k, g_i^*\rangle} \FF_k \textnormal{ for } 1\leq i\leq n,\\
    \EE_{n+2}\cdot \EE_i &= \sum_{k=1}^{m+n-1} \overline{\langle e_2f_k, g_i^*\rangle} \FF_k \textnormal{ for } 1\leq i\leq n,\\
    \EE_{n+1}\cdot \FF_i &= \sum_{k=3}^{m} \overline{\langle e_1e_k, f_i^*\rangle} \GG_k \textnormal{ for } 1\leq i\leq m+n-1, \textnormal{ and}\\
    \EE_{n+2}\cdot \FF_i &= \sum_{k=3}^{m} \overline{\langle e_2e_k, f_i^*\rangle} \GG_k \textnormal{ for } 1\leq i\leq m+n-1.
\end{align*}
\end{lemma}

\begin{proof}
Since $\rank(\phi_1\otimes k)=2$, we have that $\phi_1(u_1) = e_1$ and $\phi_1(u_2) = e_2$, with $\phi_1(u_3) \in\mathfrak{m}F_1$. To obtain a minimal resolution $C_\bullet$ from $D_\bullet$, we must split off basis elements from $D_3$ and $D_2$. Namely, we split off $G_1$ and $G_2$ from $D_3$, as these correspond with $e_1$ and $e_2$, respectively. Since $\psi_1(e_1^*)=v_{2,3}$ and $\psi_1(e_2^*)=v_{1,3}$, we split off $F_{m+n}$ and $F_{m+n+1}$ from $D_2$, as these correspond to $v_{2,3}$ and $v_{1,3}$, respectively. Thus we only consider multiplication between the basis elements $\EE_1,\dots,\EE_{n+3}$, $\FF_1,\dots,\FF_{m+n-1},\FF_{m+n+2}$, and $\GG_3,\dots,\GG_m$ in $B_\bullet$. Notice that the absence of $\GG_1$ and $\GG_2$ results in summation index starting at $k=3$ for product e. We consider each product type individually.
\begin{enumerate}[label=\alph*.]
    \item The products $\EE_{n+2}\cdot \EE_{n+3} = \FF_{m+n}$ and $\EE_{n+1}\cdot \EE_{n+3} = \FF_{m+n+1}$ have split off and are not in $B_\bullet$. The product $\EE_{n+1}\cdot \EE_{n+2} = \FF_{m+n+2}$ is in $B_\bullet$.
    \item Any products of this type involving $\EE_{n+3}$ will depend on $\phi_1(u_3)$, which is in $\mathfrak{m}F_1$ and is therefore zero in $B_\bullet$. Thus we have possible nontrivial products 
    \begin{align*}
        \EE_{n+1}\cdot \EE_i &= \sum_{k=1}^{m+n-1} \overline{\langle e_1f_k, g_i^*\rangle} \FF_k \textnormal{ for } 1\leq i\leq n \textnormal{ and}\\
        \EE_{n+2}\cdot \EE_i &= \sum_{k=1}^{m+n-1} \overline{\langle e_2f_k, g_i^*\rangle} \FF_k \textnormal{ for } 1\leq i\leq n.
    \end{align*}
    \item For the same reasoning as \Cref{lem:up4}, there are no products of this form in $B_\bullet$.
    \item The product $E_{n+3}\cdot F_{m+n+2} = d_1$ is in $\mathfrak{m}F_1^*$ and is therefore zero in $B_\bullet$.
    \item Any products of this type involving $E_{n+3}$ will depend on $\phi_1(u_3)$, which is in $\mathfrak{m}F_1$ and is therefore zero in $B_\bullet$. Thus we have possible nontrivial products
    \begin{align*}
        \EE_{n+1}\cdot \FF_i &= \sum_{k=3}^{m} \overline{\langle e_1e_k, f_i^*\rangle} \GG_k \textnormal{ for } 1\leq i\leq m+n-1 \textnormal{ and} \\
        \EE_{n+2}\cdot \FF_i &= \sum_{k=3}^{m} \overline{\langle e_2e_k, f_i^*\rangle} \GG_k \textnormal{ for } 1\leq i\leq m+n-1.
    \end{align*}
    \item The products $\EE_i\cdot \FF_{m+n+2}$ for $1\leq i\leq n$ will depend on  $\phi_2(v_{1,2}) = \phi_1(u_1)\cdot\phi_1(u_2)=e_1\cdot e_2$, which is in $\mathfrak{m}F_2$ and is therefore zero in $B_\bullet$.
    \item For the same reasoning as \Cref{lem:up4}, there are no products of this form in $B_\bullet$. \qedhere
\end{enumerate}
\end{proof}

\begin{lemma}\label{lem:over2}
Adopt the notation and hypotheses as in \Cref{setup:LinkClass}.
Consider an ideal $J$, linked to $I$, such that $\rank(\phi_1\otimes k)=2$ and $\rank(\phi_2\otimes k)=1$. Then $B_\bullet$ has guaranteed nontrivial product $\EE_{n+1}\cdot \EE_{n+2} = \FF_{m+n+2}$, as well as possible nontrivial products 
\begin{align*}
    \EE_{n+1}\cdot \EE_i &= \sum_{k=2}^{m+n-1} \overline{\langle e_1f_k, g_i^*\rangle} \FF_k \textnormal{ for } 1\leq i\leq n,\\
    \EE_{n+2}\cdot \EE_i &= \sum_{k=2}^{m+n-1} \overline{\langle e_2f_k, g_i^*\rangle} \FF_k \textnormal{ for } 1\leq i\leq n,\\
    \EE_{n+1}\cdot \FF_i &= \sum_{k=3}^{m} \overline{\langle e_1e_k, f_i^*\rangle} \GG_k \textnormal{ for } 2\leq i\leq m+n-1, \\
    \EE_{n+2}\cdot \FF_i &= \sum_{k=3}^{m} \overline{\langle e_2e_k, f_i^*\rangle} \GG_k \textnormal{ for } 2\leq i\leq m+n-1, \textnormal{ and}\\
    \EE_i\cdot \FF_{m+n+2} &= \sum_{k=3}^{m} \overline{\langle f_1e_k, g_i^*\rangle} \GG_k \textnormal{ for } 1\leq i\leq n.
\end{align*}
\end{lemma}

\begin{proof}
Since $\rank(\phi_1\otimes k)=2$ and $\rank(\phi_2\otimes k)=1$, we have that $\phi_1(u_1) = e_1$, $\phi_1(u_2) = e_2$, $\phi_1(u_3) \in\mathfrak{m}F_1$, and $e_1e_2=f_1$. To obtain a minimal resolution $C_\bullet$ from $D_\bullet$, we must split off basis elements from $D_3$, $D_2$, and $D_1$. Namely, we split off $G_1$ and $G_2$ from $D_3$, as these correspond with $e_1$ and $e_2$, respectively. Since $\psi_1(e_1^*)=v_{2,3}$ and $\psi_1(e_2^*)=v_{1,3}$, we split off $F_{m+n}$ and $F_{m+n+1}$ from $D_2$, as these correspond to $v_{2,3}$ and $v_{1,3}$, respectively. Likewise, we split off $F_1$ from $D_2$, as this corresponds to $f_1$, and we split off $E_{n+3}$ from $D_1$, since this corresponds to $u_3$ and $\psi_2(f_1^*)=u_3$. Thus we only consider multiplication between the basis elements $\EE_1,\dots,\EE_{n+2}$, $\FF_2,\dots,\FF_{m+n-1},\FF_{m+n+2}$, and $\GG_3,\dots,\GG_m$ in $B_\bullet$. Notice that there are no products of type d due to the absence of $\FF_{m+n}$, $\FF_{m+n+1}$, and $\EE_{n+3}$. Moreover, the absence of $\FF_1$ results in summation index starting at $k=2$ for product b and the absence $\GG_1$ and $\GG_2$ results in summation index starting at $k=3$ for products e and f. We consider each multiplication type individually.
\begin{enumerate}[label=\alph*.]
    \item The product $\EE_{n+1}\cdot \EE_{n+2} = \FF_{m+n+2}$ is in $B_\bullet$. 
    \item We have possible nontrivial products 
    \begin{align*}
        \EE_{n+1}\cdot \EE_i &= \sum_{k=2}^{m+n-1} \overline{\langle e_1f_k, g_i^*\rangle} \FF_k \textnormal{ for } 1\leq i\leq n \textnormal{ and}\\
        \EE_{n+2}\cdot \EE_i &= \sum_{k=2}^{m+n-1} \overline{\langle e_2f_k, g_i^*\rangle} \FF_k \textnormal{ for } 1\leq i\leq n.
    \end{align*}
    \item For the same reasoning as \Cref{lem:up4}, there are no products of this form in $B_\bullet$.
    \setcounter{enumi}{4}
    \item We have possible nontrivial products
    \begin{align*}
        \EE_{n+1}\cdot \FF_i &= \sum_{k=3}^{m} \overline{\langle e_1e_k, f_i^*\rangle} \GG_k \textnormal{ for } 2\leq i\leq m+n-1 \textnormal{ and} \\
        \EE_{n+2}\cdot \FF_i &= \sum_{k=3}^{m} \overline{\langle e_2e_k, f_i^*\rangle} \GG_k \textnormal{ for } 2\leq i\leq m+n-1.
    \end{align*}
    \item The products $\EE_i\cdot \FF_{m+n+2}$ for $1\leq i\leq n$ will depend on  $\phi_2(v_{1,2}) = \phi_1(u_1)\cdot\phi_1(u_2) = e_1e_2 = f_1$, so we have possible nontrivial products 
    \[
        \EE_i\cdot \FF_{m+n+2} = \sum_{k=3}^{m} \overline{\langle f_1e_k, g_i^*\rangle} \GG_k \textnormal{ for } 1\leq i\leq n
    \]
    \item For the same reasoning as \Cref{lem:up4}, there are no products of this form in $B_\bullet$. \qedhere
\end{enumerate}
\end{proof}

\begin{lemma}\label{lem:over3}
Adopt the notation and hypotheses as in \Cref{setup:LinkClass}.
Consider an ideal $J$, linked to $I$, such that $\rank(\phi_1\otimes k)=3$ and $\rank(\phi_2\otimes k)=0$. Then the only possible nontrivial products in $B_\bullet$ are
\begin{align*}\label{prod:over3}
    \EE_{n+1}\cdot \EE_i &= \sum_{k=1}^{m+n-1} \overline{\langle e_1f_k, g_i^*\rangle} \FF_k \textnormal{ for } 1\leq i\leq n,\\
    \EE_{n+2}\cdot \EE_i &= \sum_{k=1}^{m+n-1} \overline{\langle e_2f_k, g_i^*\rangle} \FF_k \textnormal{ for } 1\leq i\leq n,\\
    \EE_{n+3}\cdot \EE_i &= \sum_{k=1}^{m+n-1} \overline{\langle e_3f_k, g_i^*\rangle} \FF_k \textnormal{ for } 1\leq i\leq n,\\
    \EE_i\cdot \EE_j &= \overline{X(e_1\wedge e_2\wedge e_3\otimes g_i^*\wedge g_j^*)} \textnormal{ for } 1\leq i,j\leq n,\\
    \EE_{n+1}\cdot \FF_i &= \sum_{k=4}^{m} \overline{\langle e_1e_k, f_i^*\rangle} \GG_k \textnormal{ for } 1\leq i\leq m+n-1,\\
    \EE_{n+2}\cdot \FF_i &= \sum_{k=4}^{m} \overline{\langle e_2e_k, f_i^*\rangle} \GG_k \textnormal{ for } 1\leq i\leq m+n-1,\\
    \EE_{n+3}\cdot \FF_i &= \sum_{k=4}^{m} \overline{\langle e_3e_k, f_i^*\rangle} \GG_k \textnormal{ for } 1\leq i\leq m+n-1,\textnormal{ and}\\
    \EE_i\cdot \FF_j &= \overline{Y(e_1\wedge e_2\wedge e_3\otimes g_i^*\wedge f_j^*)} \textnormal{ for } 1\leq i\leq n,\enspace 1\leq j\leq m+n-1.
\end{align*}
\end{lemma}

\begin{proof}
Since $\rank(\phi_1\otimes k)=3$, we have that $\phi_1(u_1) = e_1$, $\phi_1(u_2) = e_2$, and $\phi_1(u_3) = e_3$. To obtain a minimal resolution $C_\bullet$ from $D_\bullet$, we must split off basis elements from $D_3$ and $D_2$. Namely, we split off $G_1$, $G_2$, and $G_3$ from $D_3$, as these correspond with $e_1$, $e_2$, and $e_3$, respectively. Since $\psi_1(e_1^*)=v_{2,3}$, $\psi_1(e_2^*)=v_{1,3}$, and $\psi_1(e_3^*)=v_{1,2}$, we split off $F_{m+n}$, $F_{m+n+1}$, and $F_{m+n+2}$ from $D_2$, as these correspond with $v_{2,3}$, $v_{1,3}$, and $v_{1,2}$, respectively. Thus we only consider multiplication between the basis elements $\EE_1,\dots,\EE_{n+3}$, $\FF_1,\dots,\FF_{m+n-1}$, and $\GG_4,\dots,\GG_n$ in $B_\bullet$. Notice that there are no products of type a, d, or f due to the absence of $\FF_{m+n}$, $\FF_{m+n+1}$, and $\FF_{m+n+2}$. Moreover, the absence $\GG_1$, $\GG_2$, and $\GG_3$ results in summation index starting at $k=4$ for product e. We consider each multiplication type individually.
\begin{enumerate}[label=\alph*.]
    \setcounter{enumi}{1}
    \item We have possible nontrivial products 
    \begin{align*}
        \EE_{n+1}\cdot \EE_i &= \sum_{k=1}^{m+n-1} \overline{\langle e_1f_k, g_i^*\rangle} \FF_k \textnormal{ for } 1\leq i\leq n,\\
        \EE_{n+2}\cdot \EE_i &= \sum_{k=1}^{m+n-1} \overline{\langle e_2f_k, g_i^*\rangle} \FF_k \textnormal{ for } 1\leq i\leq n, \textnormal{ and}\\
        \EE_{n+3}\cdot \EE_i &= \sum_{k=1}^{m+n-1} \overline{\langle e_3f_k, g_i^*\rangle} \FF_k \textnormal{ for } 1\leq i\leq n.
    \end{align*}
    \item Consider the matrix representation of $\phi_1$ as
    \[ \phi_1 = \begin{bmatrix}
        1 & 0 & 0\\
        0 & 1 & 0\\
        0 & 0 & 1\\
        0 & 0 & 0\\
        \vdots & \vdots & \vdots \\
        0 & 0 & 0\\
    \end{bmatrix} \]
    where $a_{i,1}, a_{i,2}, a_{i,3}\in\mathfrak{m}$ for $3\leq i\leq m$. Therefore $(\phi_1)_{rst}\notin\mathfrak{m}$ precisely when $r=1$, $s=2$, and $t=3$, in which case $(\phi_1)_{rst}=1$. Additionally, since we can take $\im\lambda \subseteq \mathfrak{a}_1\mathfrak{a}_2\mathfrak{a}_3K_2$, where $\mathfrak{a}_j$ is the ideal generated by the entries of the $j^{th}$ column of the matrix representation of $\phi_1$, we have $\im\lambda\subseteq\mathfrak{m}C_\bullet$. Therefore the possible nontrivial products in $B_\bullet$ are
    \begin{align*}
        \EE_i\cdot \EE_j &= \overline{X(e_1\wedge e_2\wedge e_3\otimes g_i^*\wedge g_j^*)} \textnormal{ for } 1\leq i,j\leq n
    \end{align*}
    \setcounter{enumi}{4}
    \item We have possible nontrivial products
    \begin{align*}
        \EE_{n+1}\cdot \FF_i &= \sum_{k=4}^{m} \overline{\langle e_1e_k, f_i^*\rangle} \GG_k \textnormal{ for } 1\leq i\leq m+n-1, \\
        \EE_{n+2}\cdot \FF_i &= \sum_{k=4}^{m} \overline{\langle e_2e_k, f_i^*\rangle} \GG_k \textnormal{ for } 1\leq i\leq m+n-1, \textnormal{ and}\\
        \EE_{n+3}\cdot \FF_i &= \sum_{k=4}^{m} \overline{\langle e_3e_k, f_i^*\rangle} \GG_k \textnormal{ for } 1\leq i\leq m+n-1.
    \end{align*}
    \setcounter{enumi}{6}
    \item By the same argument in the discussion of product c above, $(\phi_1)_{rst}\notin\mathfrak{m}$ precisely when $r=1$, $s=2$, and $t=3$, in which case $(\phi_1)_{rst}=1$. Since $d_1^*(\mu(g^*\otimes f^*))\in\mathfrak{m}F_1^*$, the possible nontrivial products in $B_\bullet$ are
    \begin{align*}
        \EE_i\cdot \FF_j &= \overline{Y(e_1\wedge e_2\wedge e_3\otimes g_i^*\wedge f_j^*)} \textnormal{ for } 1\leq i\leq n, \enspace 1\leq j\leq m+n-1. \qedhere
    \end{align*}
\end{enumerate}
\end{proof}

\begin{chunk}\textbf{Proof of \Cref{linktoT}.}
According to \Cref{setup:LinkClass}, choose the regular sequence $x_1,x_2,x_3$ such that $\rank(\phi_1\otimes k)=0$. By the proof of \Cref{LinkOption}, this implies $\rank(\phi_2\otimes k)=0$. Therefore, $\phi_1(u_1), \phi_1(u_2), \phi_1(u_3) \in\mathfrak{m}F_1$. This implies products b, c, e, and g are in $\mathfrak{m}C_\bullet$. Moreover, $\phi_2(v_{1,2}), \phi_2(v_{1,3}), \phi_2(v_{2,3}) \in\mathfrak{m}F_2$, so product f is in $\mathfrak{m}C_\bullet$. Finally, $d_1$ takes image in $\mathfrak{m}$, so product d will be trivial in $C_\bullet\setminus{\mathfrak{m}C_\bullet}$. Therefore we only have nontrivial products coming from product a, which are $E_{n+1}\cdot E_{n+2} = F_{m+n+2}$, $E_{n+1}\cdot E_{n+3} = F_{m+n+1}$, and $E_{n+2}\cdot E_{n+3} = F_{m+n+2}$. Each of these products descends to $B_\bullet$, so $J$ is of class $\mathbf{T}$.
\end{chunk}

\begin{chunk}\textbf{Proof of \Cref{linkT}.}
Let $I$ be of class $\mathbf{T}$ and format $(m,n)$. First, choose a basis for $F_\bullet$ such that the products in $F_\bullet$ that induce nontrivial products in $A_\bullet$ are $e_1e_2=f_1$, $e_1e_3=f_2$, and $e_2e_3=f_3$. We prove (ii) and (iv).
\begin{enumerate}[label=(\roman*)]
    \setcounter{enumi}{1}
    \item According to \Cref{setup:LinkClass}, choose the regular sequence $x_1,x_2,x_3$ such that $\rank(\phi_1\otimes \Bbbk)=1$. By \Cref{lem:up4}, the only nontrivial products in $B_\bullet$ are  $\EE_{n+1}\cdot \EE_{n+2} = \FF_{m+n+2}$ and $\EE_{n+1}\cdot \EE_{n+3} = \FF_{m+n+1}$, as well as
    \begin{align*}
        \EE_{n+1}\cdot \FF_1 &= \sum_{k=2}^{m+n-1} \langle e_1e_k,f_1^*\rangle \GG_k = \langle e_1e_2, f_1^*\rangle \GG_2
         = \GG_2 \textnormal{ and}\\
        \EE_{n+1}\cdot \FF_2 &= \sum_{k=2}^{m+n-1} \langle e_1e_k,f_2^*\rangle \GG_k = \langle e_1e_3, f_2^*\rangle \GG_3
         = \GG_3,
    \end{align*}
    making the linked ideal class $\mathbf{H}(2,2)$.
    \setcounter{enumi}{3}
    \item According to \Cref{setup:LinkClass}, choose the regular sequence $x_1,x_2,x_3$ such that  $\rank(\phi_1\otimes \Bbbk)=2$. By choice of the products above, $\rank(\phi_2\otimes \Bbbk)=1$, so we must apply \Cref{lem:over2}. The only nontrivial products in  $B_\bullet$ are  $\EE_{n+1}\cdot \EE_{n+2} = \FF_{m+n+2}$, as well as
    \begin{align*}
        \EE_{n+1}\cdot \FF_2 &= \sum_{k=3}^{m+n-1} \langle e_1e_k,f_2^*\rangle \GG_k = \langle e_1e_3, f_2^*\rangle \GG_3
         = \GG_3 \textnormal{ and}\\
        \EE_{n+2}\cdot \FF_3 &= \sum_{k=3}^{m+n-1} \langle e_2e_k,f_3^*\rangle \GG_k = \langle e_2e_3, f_3^*\rangle \GG_3
         = \GG_3,
    \end{align*}
    making the linked ideal class $\mathbf{B}$.
\end{enumerate}

\noindent Now, choose a basis for $F_\bullet$ such that the products in $F_\bullet$ that induce nontrivial products in $A_\bullet$ are $e_2e_3=f_1$, $e_2e_4=f_2$, and $e_3e_4=f_3$. We prove (i) and (iii).
\begin{enumerate}[label=(\roman*)]
    \item According to \Cref{setup:LinkClass}, choose the regular sequence $x_1,x_2,x_3$ such that $\rank(\phi_1\otimes \Bbbk)=1$. By \Cref{lem:up4}, the only nontrivial products in $B_\bullet$ are $\EE_{n+1}\cdot \EE_{n+2} = \FF_{m+n+2}$ and $\EE_{n+1}\cdot \EE_{n+3} = \FF_{m+n+1}$, making the linked ideal class $\mathbf{H}(2,0)$.
    \setcounter{enumi}{2}
    \item According to \Cref{setup:LinkClass}, choose the regular sequence $x_1,x_2,x_3$ such that $\rank(\phi_1\otimes \Bbbk)=2$. By choice of the products above, $\rank(\phi_2\otimes \Bbbk)=0$, so we must apply \Cref{lem:up2}. The only nontrivial product in $B_\bullet$ is $\EE_{n+1}\cdot \EE_{n+2} = \FF_{m+n+2}$, as well as 
    \begin{align*}
        \EE_{n+2}\cdot \FF_1 &= \sum_{k=3}^{m+n-1} \langle e_2e_k,f_1^*\rangle \GG_k = \langle e_2e_3, f_1^*\rangle \GG_3
         = \GG_3 \textnormal{ and}\\
        \EE_{n+2}\cdot \FF_2 &= \sum_{k=3}^{m+n-1} \langle e_2e_k,f_2^*\rangle \GG_k = \langle e_2e_4, f_2^*\rangle \GG_4
         = \GG_4, 
    \end{align*}
    making the linked ideal class $\mathbf{H}(1,2)$. \qed
\end{enumerate}
\end{chunk}

\begin{chunk}\textbf{Proof of \Cref{linkG}.}
Let $I$ be of class $\mathbf{G}(r)$ and format $(m,n)$. Choose a basis for $F_\bullet$ such that the products in $F_\bullet$ that induce nontrivial products in $A_\bullet$ are $e_{i}f_{i}=g_1$ for $1\leq i\leq r$. 
\begin{enumerate}[label=(\roman*)]
    \item According to \Cref{setup:LinkClass}, choose the regular sequence $x_1,x_2,x_3$ such that $\rank(\phi_1\otimes \Bbbk)=1$. By \Cref{lem:up4}, the only nontrivial products in $B_\bullet$ are $\EE_{n+1}\cdot \EE_{n+2} = \FF_{m+n+2}$ and $\EE_{n+1}\cdot \EE_{n+3} = \FF_{m+n+1}$, as well as
    \begin{equation*}
        \EE_{n+1}\cdot \EE_1 = \sum_{k=1}^{m+n-1} \overline{\langle e_1f_k,g_1^*\rangle} \FF_k = \overline{\langle e_1f_1, g_1^*\rangle} \FF_1
         = \FF_1,
    \end{equation*}
    making the linked ideal class $\mathbf{H}(3,0)$.
\end{enumerate}

\begin{enumerate}[label=(\roman*)]
    \setcounter{enumi}{1}
    \item According to \Cref{setup:LinkClass}, choose the regular sequence $x_1,x_2,x_3$ such that $\rank(\phi_1\otimes \Bbbk)=2$. By choice of the products above, $\rank(\phi_2\otimes \Bbbk)=0$, so we must apply \Cref{lem:up2}. The only nontrivial product in $B_\bullet$ is $\EE_{n+1}\cdot \EE_{n+2} = \FF_{m+n+2}$, as well as
    \begin{align*}
        \EE_{n+1}\cdot \EE_1 &= \sum_{k=1}^{m+n-1} \overline{\langle e_1f_k,g_1^*\rangle} \FF_k = \overline{\langle e_1f_1, g_1^*\rangle} \FF_1
         = \FF_1,\textnormal{ and}\\
        \EE_{n+2}\cdot \EE_1 &= \sum_{k=1}^{m+n-1} \overline{\langle e_2f_k,g_1^*\rangle} \FF_k = \overline{\langle e_2f_2, g_1^*\rangle} \FF_2
         = \FF_2,
    \end{align*}
    making the linked ideal class $\mathbf{T}$. \qed
\end{enumerate}
\end{chunk}

\begin{chunk}\textbf{Proof of \Cref{linkH}.}
Let $I$ be of class $\mathbf{H}(p,q)$ and format $(m,n)$. First, choose a basis for $F_\bullet$ such that the products in $F_\bullet$ that induce nontrivial products in $A_\bullet$ are $e_1e_{i+1}=f_i$ for $1\leq i\leq p$ and $e_1f_{i+p}=g_i$ for $1\leq i\leq q$. 
We prove (ii).
\begin{enumerate}[label=(\roman*)]
    \setcounter{enumi}{1}
    \item According to \Cref{setup:LinkClass}, choose the regular sequence $x_1,x_2,x_3$ such that $\rank(\phi_1\otimes \Bbbk)=1$. By \Cref{lem:up4}, the only nontrivial products in $B_\bullet$ are $\EE_{n+1}\cdot \EE_{n+2} = \FF_{m+n+2}$ and $\EE_{n+1}\cdot \EE_{n+3} = \FF_{m+n+1}$, as well as
    \begin{align*}
        \EE_{n+1}\cdot \EE_i &= \sum_{k=1}^{m+n-1} \overline{\langle e_1f_k,g_i^*\rangle} \FF_k = \overline{\langle e_1f_{i+p}, g_i^*\rangle} \FF_{i+p} = \FF_{i+p} \textnormal{ for } 1\leq i\leq q \textnormal{ and}\\
        \EE_{n+1}\cdot \FF_i &= \sum_{k=2}^{m} \overline{\langle e_1e_k,f_i^*\rangle} \GG_k = \overline{\langle e_1e_{i+1}, f_i^*\rangle} \GG_{i+1} = \GG_{i+1} \textnormal{ for } 1\leq i\leq p,
    \end{align*}
    making the linked ideal class $\mathbf{H}(q+2,p)$.
\end{enumerate}

\noindent Now, choose a basis for $F_\bullet$ such that the products in $F_\bullet$ that induce nontrivial products in $A_\bullet$ are $e_1e_{i+2}=f_i$ for $1\leq i\leq p$ and $e_1f_{i+p}=g_i$ for $1\leq i\leq q$. Since $0\leq p\leq m-2$, we can choose the products above such that $e_2$ is not involved in the product structure. We prove (iv).
\begin{enumerate}[label=(\roman*)]
    \setcounter{enumi}{3}
    \item According to \Cref{setup:LinkClass}, choose the regular sequence $x_1,x_2,x_3$ such that  $\rank(\phi_1\otimes \Bbbk)=2$. By choice of the products above, $\rank(\phi_2\otimes \Bbbk)=0$, so we must apply \Cref{lem:up2}. The only nontrivial product in $B_\bullet$ are $\EE_{n+1}\cdot \EE_{n+2} = \FF_{m+n+2}$, as well as
    \begin{align*}
        \EE_{n+1}\cdot \EE_i &= \sum_{k=1}^{m+n-1} \overline{\langle e_1f_k,g_i^*\rangle} \FF_k = \overline{\langle e_1f_{i+p}, g_i^*\rangle} \FF_{i+p} = \FF_{i+p} \textnormal{ for } 1\leq i\leq q \textnormal{ and}\\
        \EE_{n+1}\cdot \FF_i &= \sum_{k=3}^{m} \overline{\langle e_1e_k,f_i^*\rangle} \GG_k = \overline{\langle e_1e_{i+2}, f_i^*\rangle} \GG_{i+2} = \GG_{i+2} \textnormal{ for } 1\leq i\leq p,
    \end{align*}
    making the linked ideal class $\mathbf{H}(q+1,p)$.
\end{enumerate}

\noindent Now, choose a basis for $F_\bullet$ such that the products in $F_\bullet$ that induce nontrivial products in $A_\bullet$ are $e_1e_{i+3}=f_i$ for $1\leq i\leq p$ and $e_1f_{i+p}=g_i$ for $1\leq i\leq q$. Since $2\leq p\leq m-3$, we can choose the products above such that $e_2$ and $e_3$ are not involved in the product structure. We prove (v).
\begin{enumerate}[label=(\roman*)]
    \setcounter{enumi}{4}
    \item According to \Cref{setup:LinkClass}, choose the regular sequence $x_1,x_2,x_3$ such that  $\rank(\phi_1\otimes k)=3$. By choice of the products above, $\rank(\phi_2\otimes k)=0$. By \cite[Proposition 3.2(b)]{CVW20Linkage}, $J$ is of class $\mathbf{H}(0,q)$, where $q\geq p$. Therefore $B_\bullet$ has no products of type a--c and by \Cref{lem:over3}, $B_\bullet$ has the nontrivial products
    \begin{align*}
        \EE_{n+1}\cdot \FF_i &= \sum_{k=3}^{m} \overline{\langle e_1e_k,f_i^*\rangle} \GG_k = \overline{\langle e_1e_{i+3}, f_i^*\rangle} \GG_{i+3} = \GG_{i+3} \textnormal{ for } 1\leq i\leq p,
    \end{align*}
    as well as the possibly nontrivial products
    \begin{align*}
        \EE_i\cdot \FF_j &= \overline{Y(e_1\wedge e_2\wedge e_3\otimes g_i^*\wedge f_j^*)} \textnormal{ for } 1\leq i\leq n, 1\leq j\leq m+n-1.
    \end{align*}
    Recall that $q=\dim A_1A_2$. To show that $q=p$, we will show that the products $\EE_i\FF_j$ are in the subspace of $A_3$ spanned by $\GG_4,\GG_5,\dots,\GG_{p+3}$. To do so, we consider the subspace $U$ generated by $\EE_{n+1}$ and $\EE_i$ for some fixed $1\leq i\leq n$. Since $J$ is of class $\mathbf{H}(0,q)$ with $q\geq p\geq 2$, it follows that $U$ must have multiplicative structure according to \cite[Lemma 2.5(a)]{AKM88}. In this case, we may choose a basis for $B_\bullet$ in which only one basis element of $U$ contributes to the product structure. The basis elements of $U$ are of the form $\mathcal{E}_1=a\EE_{n+1}+b\EE_i$ and $\mathcal{E}_2=c\EE_{n+1}+d\EE_i$. Without loss of generality, assume that $\mathcal{E}_1$ contributes to the product structure and $\mathcal{E}_2$ does not. This implies $d\neq 0$; otherwise, $\mathcal{E}_2=cE_{n+1}$ would contribute to the product structure. Since $\mathcal{E}_2$ does not contribute to the product structure, we have 
    \begin{align*}
        0 &= \mathcal{E}_2 \FF_1 = (c\EE_{n+1}+d\EE_i)\FF_1 = c \EE_{n+1}\FF_1 + d \EE_i\FF_1 = c\GG_4 + d\EE_i\FF_1,\\
        0 &= \mathcal{E}_2 \FF_2 = (c\EE_{n+1}+d\EE_i)\FF_2 = c \EE_{n+1}\FF_2 + d \EE_i\FF_2 = c\GG_5 + d\EE_i\FF_2, \\
        \vdots\\
        0 &= \mathcal{E}_2 \FF_p = (c\EE_{n+1}+d\EE_i)\FF_p = c \EE_{n+1}\FF_p + d \EE_i\FF_p = c\GG_{p+3} + d\EE_i\FF_p, \\
        0 &= \mathcal{E}_2 \FF_{p+1} = (c\EE_{n+1}+d\EE_i)\FF_{p+1} = c \EE_{n+1}\FF_{p+1} + d \EE_i\FF_{p+1} = d\EE_i\FF_{p+1},\\
        \vdots\\
        0 &= \mathcal{E}_2 \FF_{m+n-1} = (c\EE_{n+1}+d\EE_i)\FF_{m+n-1} = c \EE_{n+1}\FF_{m+n-1} + d \EE_i\FF_{m+n-1} = d\EE_i\FF_{m+n-1}.
    \end{align*}
    This implies that $\EE_i\FF_j=-cd^{-1}\GG_{j+3}$ for $1\leq j\leq p$ and $\EE_i\FF_j=0$ for $p+1\leq j\leq m+n-1$. Therefore, the products involving $\EE_i$ are contained in the subspace spanned by $\GG_4,\GG_5,\dots,\GG_{p+3}$ and $q$ does not increase a result of these products. Since $i$ is arbitrary, this argument holds for all $1\leq i\leq n$ and we can conclude that $q=p$ and $J$ is of class $\mathbf{H}(0,p)$.
\end{enumerate}

\noindent Now, choose a basis for $F_\bullet$ such that the products in $F_\bullet$ that induce nontrivial products in $A_\bullet$ are $e_2e_1=f_1$, $e_2e_{i+1}=f_{i}$ for $2\leq i\leq p$, and $e_2f_{i+p}=g_i$ for $1\leq i\leq q$. We assume $p\geq 1$ and prove (i).
\begin{enumerate}[label=(\roman*)]
    \item According to \Cref{setup:LinkClass}, choose the regular sequence $x_1,x_2,x_3$ such that  $\rank(\phi_1\otimes\Bbbk)=1$. By \Cref{lem:up4}, the only nontrivial products in $B_\bullet$ are $\EE_{n+1}\cdot \EE_{n+2} = \FF_{m+n+2}$ and $\EE_{n+1}\cdot \EE_{n+3} = \FF_{m+n+1}$, as well as
    \begin{align*}
        \EE_{n+1}\cdot \FF_1 &= \sum_{k=2}^{m} \overline{\langle e_1e_k,f_1^*\rangle} \GG_k = \overline{\langle e_1e_2, f_1^*\rangle} \GG_2 = -\GG_2, 
    \end{align*}
    making the linked ideal class $\mathbf{H}(2,1)$.
\end{enumerate}

\noindent Finally, choose a basis for $F_\bullet$ such that the products in $F_\bullet$ that induce nontrivial products in $A_\bullet$ are $e_3e_1=f_1$, $e_3e_{i+2}=f_{i}$ for $2\leq i\leq p$, and $e_3f_{i+p}=g_i$ for $1\leq i\leq q$. Since $1\leq p\leq m-2$, we can choose the products above such that $e_2$ is not involved in the product structure. We prove (iii).
\begin{enumerate}[label=(\roman*)]
    \setcounter{enumi}{2}
    \item According to \Cref{setup:LinkClass}, choose the regular sequence $x_1,x_2,x_3$ such that  $\rank(\phi_1\otimes k)=2$. By choice of the products above, $\rank(\phi_2\otimes k)=0$, so we must apply \Cref{lem:up2}. The only nontrivial product in $B_\bullet$ is $\EE_{n+1}\cdot \EE_{n+2} = \FF_{m+n+2}$, as well as
    \begin{align*}
        \EE_{n+1}\cdot \FF_1 &= \sum_{k=3}^{m} \overline{\langle e_1e_k,f_1^*\rangle} \GG_k = \overline{\langle e_1e_3, f_1^*\rangle} \GG_3 = -\GG_3,
    \end{align*}
    making the linked ideal class $\mathbf{H}(1,1)$. \qed
\end{enumerate}
\end{chunk}

\bibliographystyle{amsplain}
\bibliography{biblio}
\end{document}